\documentclass[reqno,11pt, twoside]{amsart}
\numberwithin{equation}{section}
\usepackage{enumitem}	     
\usepackage{geometry}               
\usepackage[font={scriptsize}]{caption} 
\usepackage{graphicx}
\usepackage{amssymb, amsmath, amscd, amsthm,xcolor,mathrsfs}
\usepackage{epstopdf}
\usepackage{subfig}
\newcommand{\be}{\begin{equation} }
\newcommand{\ee}{\end{equation}}
\newcommand{\bse}{\begin{subequations}}
\newcommand{\ese}{\end{subequations}}

\newcommand{\bc}{\mathrm{bc}}
\newcommand{\jbracket}[1]{\langle{#1}\rangle}
\newcommand{\Uybc}{(\partial_{x_2} U)_{\mathrm{bc}}}
\newcommand{\Gammas}{\Gamma_{\mathrm{s}}}

\newcommand{\Gammast}{\tilde{\Gamma}_{\mathrm{s}}}

\newcommand{\Gammasc}{\check{\Gamma}_{\mathrm{s}}}

\newcommand{\vecbundlek}{\mathscr{X}_{\delta, k}}
\newcommand{\vecbundlevGk}{\mathscr{W}_{\delta, k}}

\newcommand{\p}{\partial}

\newcommand{\placeholder}{\,\cdot\,}

\newcommand{\C}{\mathbb C}
\newcommand{\I}{\mathrm i}
\newcommand{\R}{\mathbb R}

\newcommand{\diff}{{\mathrm d}}
\newcommand{\Diff}{{\mathrm D}}
\newcommand{\e}{{\mathrm e}}

\newcommand{\slab}{\mathcal{S}}

\DeclareMathOperator{\dist}{dist}
\DeclareMathOperator{\im}{Im}

\DeclareMathOperator{\F}{{\mathcal F}}

\DeclareMathOperator{\id}{id}

\DeclareMathOperator{\spn}{span}
\DeclareMathOperator{\sign}{sgn}
\DeclareMathOperator{\graph}{graph}

\theoremstyle{plain} 
\newtheorem{theorem}{Theorem}[section] 
\newtheorem{corollary}[theorem]{Corollary}
 
\newtheorem{lemma}[theorem]{Lemma}
\newtheorem{proposition}[theorem]{Proposition}

\theoremstyle{remark}
\newtheorem{remark}[theorem]{Remark}

\title[Water waves with localized vorticity]{Smooth stationary water waves with exponentially localized vorticity}

\author[M. Ehrnstr\"om]{Mats Ehrnstr\"om}

\address[M. Ehrnstr\"om]{Department of Mathematical Sciences, Norwegian University of Science and Technology, 7491 Trondheim, Norway}
\email{mats.ehrnstrom@ntnu.no}
\thanks{ME was supported by grant nos. 231668 and 250070 from the Research Council of Norway.}

\author[S. Walsh]{Samuel Walsh}
\address[S. Walsh]{Department of Mathematics, University of Missouri, Columbia, MO 65211} 
\email{walshsa@missouri.edu} 
\thanks{SW was supported by the National Science Foundation through the awards DMS-1514910 and DMS-1812436.} 

\author[C. Zeng]{Chongchun Zeng}
\address[C. Zeng]{School of Mathematics, Georgia Institute of Technology, Atlanta, GA 30332}
\email{zengch@math.gatech.edu}
\thanks{CZ was supported in part by the National Science Foundation through the award DMS-1900083.} 

\usepackage{accents}

\usepackage[colorlinks=true]{hyperref}
\hypersetup{urlcolor=blue, citecolor=red, linkcolor=blue}

\begin{document}

\begin{abstract}
We study stationary capillary-gravity waves in a two-dimensional body of water that rests above a flat ocean bed and below vacuum.  
This system is described by the Euler equations with a free surface.  A great deal of recent activity has focused on finding waves with nontrivial vorticity $\omega$. There are now many results on the existence of solutions to this problem for which the vorticity is non-vanishing at infinity, and several authors have constructed waves with $\omega$ having compact support.       Our main theorem states that there are large families of stationary capillary-gravity waves that carry finite energy and exhibit an exponentially localized distribution of vorticity.  They are solitary waves in the sense that the free surface is asymptotically flat.  Remarkably, while their amplitude is small, the kinetic energy is $O(1)$.  In this and other respects, they are strikingly different from previously known rotational water waves.  

To construct these solutions, we exploit a previously unobserved connection between the steady water wave problem on the one hand and singularly perturbed elliptic PDE on the other.  Indeed, our result expands the study of spike-layer solutions to free boundary problems with physical relevance.  
\end{abstract}
\maketitle

\section{Introduction} \label{sec:formulation}

We consider waves in a two-dimensional body of water that has finite depth.  Mathematically, they are modeled as solutions to the incompressible Euler equation
\bse \label{SE:Euler}
\be \label{E:Euler}
\p_t v + (v \cdot \nabla) v + \nabla p + g e_2 =0,  
\ee  
on the evolving fluid domain 
\be \label{E:fluidD}
\Omega (t) = \left \{ (x_1, x_2) \in \mathbb{R}^2 \colon -1 < x_2 < 1 +  \eta(t, x_1) \right\}.
\ee
Here, the water density is assumed to be of constant value $1$, $v = v(t,\placeholder) \colon \Omega(t) \to \mathbb{R}^2$ is the velocity, $p = p(t, \placeholder) \colon \Omega(t) \to \mathbb{R}$ is the pressure, $g > 0$ is the constant gravitational acceleration, and $e_2=(0,1)$.  Notice that the water is bounded below by a rigid and perfectly flat bed at $\{x_2 = -1\}$.  The upper boundary, given by the graph of $1+\eta$, represents the interface between the water and a region of air which is treated as vacuum.  An important feature of this problem is that $\eta$ is one of the unknowns in the system.  For solitary waves, $\eta$ vanishes as $|x_1| \to \infty$, and hence the asymptotic depth is normalized to be $2$.

The kinematic boundary conditions state that the velocity field does not penetrate the bed:
\be \label{E:BC1}
v_2 =0 \qquad \textrm{on }  x_2 = -1, 
\ee 
and, along the free surface, we have
\be \label{E:kinematic}
\p_t \eta 
= - v_1 \p_{x_1} \eta  + v_2 \qquad \textrm{on }  x_2 = 1 + \eta(t, x_1).
\ee
Lastly, on the surface we impose the dynamic condition according to the Young--Laplace law that 
\be \label{E:pressure}
p = \alpha^2 \kappa \qquad \textrm{on }  x_2 = 1 + \eta(t,x_1),
\ee
\ese
where $\alpha>0$ is a constant measuring the surface tension and 
\be \label{E:curvature}
\kappa = - \frac {\p_{x_1}^2 \eta }{\left(1+(\p_{x_1} \eta)^2 \right)^{\frac 32}}
\ee
is the signed curvature.  
Because $g, \alpha > 0$, we always presume that surface tension is present on the interface and that gravity acts in the bulk.  Solutions of \eqref{SE:Euler} are therefore called \emph{capillary-gravity waves}. 

A stationary water wave is a solution to \eqref{SE:Euler} that is independent of time.  More generally, one can consider steady or traveling waves, which are solutions that become time independent after shifting to a moving frame of reference.  These are among the oldest and most important examples of nonlinear wave phenomena studied in mathematics.  

Perhaps the central object of interest for this paper is the \emph{vorticity}
\be \label{definition omega}
\omega :=\nabla^\perp \cdot v= \partial_{x_1} v_2 - \partial_{x_2} v_1.
\ee
which is the third component of $\nabla \times (v_1, v_2, 0)$. 
The earliest rigorous constructions of steady water waves were given by Levi-Civita \cite{levicivita1925determination} and Nekrasov \cite{nekrasov1951exact}, who worked in the irrotational regime where $\omega$  vanishes identically.  This assumption permits several elegant reformulations of the problem that are far more tractable; see, for example, the survey \cite{toland1996stokes}.  However, beginning in the early 2000s, substantial inroads have been made in the rigorous analysis of rotational water waves. With a few exceptions, these results pertain to waves without interior stagnation, meaning that the streamlines (the integral curves of $v$) are never closed, and hence the vorticity does not vanish at infinity. 

In practice, though, many of the effects that generate vorticity are \emph{local} --- wind blowing over a section of the water or a boundary layer caused by an immersed body, for example.   This naturally leads us to seek waves for which $\omega$ is concentrated in the near field.  A completely different analytical approach is necessary to treat this situation, however.   Consequently, there are comparatively very few rigorous results for waves with localized vorticity, and those that are available concern either periodic waves or waves with compactly supported vorticity; see the overview below.  

Another important quantity associated to the system is the total energy $E$ defined by
\be \label{E:energy}
E = \frac 12 \int_\Omega  |v|^2 \, \diff x + \int_\R \frac 12 g \eta^2 + \alpha^2 \left(\sqrt{1+(\p_{x_1} \eta)^2} -1\right)  \, \diff x_1.
\ee 
The first term on the right-hand side represents the kinetic energy, while the second is gravitational potential energy, and the third is the surface energy.   It is well-known that $E$ is conserved by sufficiently smooth solutions of the time-dependent problem.  It is physically desirable, therefore, to construct waves that carry a finite amount of total energy, which in particular means that $v$ must be in $L^2(\Omega)$.

As the main contribution of this paper, we prove the existence of large families of solitary stationary water waves with a \emph{smooth}, \emph{highly localized vorticity} and a \emph{finite energy} $E<\infty$: in a perturbed disk around the origin the vorticity is large and negative, and outside it is positive and exponentially decaying.  Qualitatively, this represent an entirely novel species of water wave that we call a \emph{vortex spike}.   Our method establishes a connection between singularly perturbed elliptic equations and physical problems with free boundaries.  This application to water waves is at once quite natural and yet completely new.

\subsection{Main theorem}
We now state the result more precisely.  In two dimensions, divergence free vector fields can be represented through a stream function, namely, 
\[
v = \nabla^\perp \Psi :=  (-\p_{x_2}\Psi, \p_{x_1} \Psi).
\]
One can easily confirm that from \eqref{definition omega} that $\omega = \Delta \Psi$.

As mentioned above, our interest is in \emph{smooth finite energy stationary} waves with \emph{spatially highly localized vorticity}. 
For the momentum equation \eqref{E:Euler}, we see $\omega$ satisfies
\be \label{vorticity equation}
 \p_t \omega + v \cdot \nabla \omega =0 \qquad \textrm{in } \Omega(t),
 \ee
and hence the vorticity is transported by the Lagrangian flow.  In terms of the stream function, for the stationary case this becomes
\begin{equation}\label{eq:momentum}  
 \nabla^\perp \Psi \cdot \nabla \Delta \Psi  = v \cdot \nabla \omega = 0 \qquad \textrm{in } \Omega.
\end{equation} 
The kinematic boundary conditions \eqref{E:BC1}--\eqref{E:kinematic} imply that $\Psi$ is a constant along each component of $\p \Omega$.  Without loss of generality, we take \be \label{E:SKinematic}
\Psi|_{\p \Omega}=0;
\ee
see Section \ref{sec:intro:motivation} for more discussion about this.  At the same time, the dynamic condition \eqref{E:pressure} can be expressed in terms of $\Psi$ as the well-known Bernoulli equation
\begin{equation}\label{eq:dynamic}
{\textstyle \frac{1}{2}} | \nabla \Psi|^2 + g x_2 + \alpha^2 \kappa = g \qquad \text{on } x_2 = 1 + \eta(x_1).
\end{equation}

Together, \eqref{eq:momentum}--\eqref{eq:dynamic} are equivalent to the (stationary) Euler equations \eqref{SE:Euler}.  We seek to construct waves for which the stream function and the vorticity will have the leading order forms
\be \label{eq:def psi}
\Psi(x) = U\left(\frac {x-x_*}\delta\right) + \ldots \in H^k (\Omega) \cap H_0^1 (\Omega), \quad \omega (x) = \frac 1{\delta^2} \Delta U\left(\frac {x-x_*}\delta\right) + \cdots, 
\ee
where $0< \delta \ll1$, \(x_*\) is roughly the location of the vorticity to be determined in the proof which will turn out to be very close to the origin in out coordinate system, and $U$ is a smooth solution to \eqref{eq:momentum} on the whole of $\R^2$, exponentially decaying as $|x| \to \infty$. It is well known that \eqref{eq:momentum} is satisfied provided that $\omega = \gamma(\Psi)$, for some \emph{vorticity function} $\gamma$.  We therefore construct $\Psi$ as the solution to 
\be \label{eq:main} 
\Delta \Psi = \frac 1{\delta^2} \gamma(\Psi) \qquad\text{ in }  \Omega, 
\ee
with $U$ a solution to
\begin{equation}\label{eq:ground state}
\Delta U = \gamma(U) \qquad\text{ in }  \R^2.
\end{equation}
We will assume that  $\gamma$ satisfies the following. \\[-8pt] 
\begin{enumerate}[label=\rm(\Alph*)]
\item \label{assumption:A}  
\(\gamma \in C^{k_0}(\R, \R)\), $k_0\ge 2$, $\gamma(0)=0$, $\gamma'(0)=1$, and 
\eqref{eq:ground state} has a nontrivial radial solution \(U \in C^{k_0+2}(\R^2)\) satisfying $U(x), \nabla U(x) \to 0$ as $|x| \to \infty$, and \\[-8pt] 
\item \label{assumption:B}  
the kernel of $-\Delta + \gamma^\prime(U)\colon H^2(\mathbb{R}^2) \to L^2(\mathbb{R}^2)$ is 
equal to $\spn\{\partial_{x_1} U, \partial_{x_2} U\}$. \\[-8pt]
\end{enumerate}
We would like to point out that the asymptotic vanishing of $U$ and $\nabla U$ at $|x|= \infty$ can be ensured by further asking that $U \in L^2 (\R^2) \cap L^\infty (\R^2)$ or $U \in H^1 (\R^2)$, see Remark \ref{R:UDecay}.  Also, as a consequence of \ref{assumption:A}, $U$ and its derivatives up to order $k_0+1$ decay exponentially, and so in particular, $U \in H^{k_0+2} (\R^2)$, see Proposition \ref{prop:sign of Uy}. Prototypical functions $\gamma$ fulfilling assumptions~\ref{assumption:A} and~\ref{assumption:B} are \(\gamma(t) = t - |t|^p t\), for integers \(p \geq 1\), 
but many others will do as well. Classical results for dimension \(n=2\) may be found in for example \cite{MR734575, MR695535, MR969899}, and a modern summary including the non-degeneracy results in \cite{MR3626577}. 

Under the above assumptions, our main theorem is as follows.

\begin{theorem} \label{main euler theorem}  
For any $\gamma$ as in Assumptions~\ref{assumption:A} and~\ref{assumption:B},  there exists $\delta_0 > 0$ such that, for each $\delta \in (0,\delta_0)$, there is a finite energy solution 
\[ 
 (\Psi, \eta) \in \left( H_0^1(\Omega) \cap H^{k_0} (\Omega) \right)  \times H^{k_0}(\mathbb{R})
\]
to the stationary water wave problem \eqref{E:SKinematic}, \eqref{eq:dynamic}, and \eqref{eq:main}.  Both $\Psi$ and $\eta$ are even in \(x_1\). Moreover, there exists a constant $C>0$, independent of $\delta$ but depending on $\gamma$, such that for each $\delta \in (0, \delta_0)$ there exists $\tau$ with $|\tau| \le C\delta^{-\frac 72} e^{-\frac 2{\delta}}$ satisfying 
\be 
|\Psi -  \Psi_0|_{H^{k_0} (\Omega)} \le C \delta^{1 -2k_0} e^{-\frac 2\delta},  
\label{Psi asymptotics} \ee
where 
\[ 
\Psi_0 (x)  = U\left(\frac {x_1}\delta, \frac {x_2 -\tau}\delta \right) - U\left(\frac {x_1}\delta, \frac {2- x_2 -\tau}\delta \right) - U\left(\frac {x_1}\delta, \frac {-2- x_2 -\tau}\delta \right), 
\]
and
\be 
|\eta|_{H^{k_0} (\R)} \le C \delta^{1 - k_0} e^{-\frac 2\delta}, 
\quad |\eta - \eta_0|_{H^{k_0} (\R)} \le C \delta^{\frac 34 -2k_0}  e^{-\frac {3}{\delta}},
\label{eta asymptotics} \ee
with 
\[
\eta_0 = - 2\delta^{-2} (g-\alpha^2 \p_{x_1}^2)^{-1} \left( \left( \p_{x_2} U(\tfrac \placeholder \delta, \tfrac 1\delta)\right)^2 \right) =- \frac 1{\alpha \sqrt{g} \delta^{2}} e^{-\frac {\sqrt{g}}\alpha |\placeholder|} *   \left( \left( \p_{x_2} U(\tfrac \placeholder\delta, \tfrac 1\delta)\right)^2 \right). 
\]
\end{theorem}

\begin{centering} 
\begin{figure} 
\includegraphics[width=.7\linewidth]{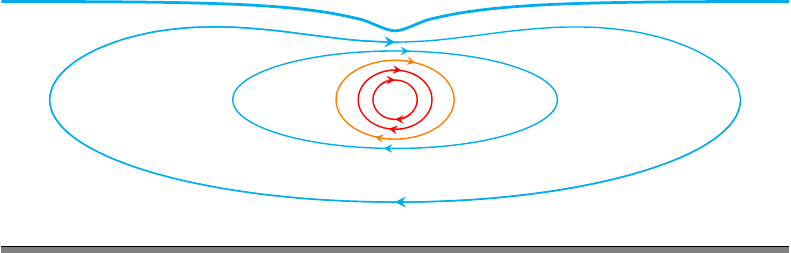}
\caption{Schematic representation of the streamline pattern and free surface.  Blue lines indicate positive vorticity, red is negative, and orange is zero.  Note that there is a critical layer and all streamlines are closed except the boundary components.  For stationary waves, the fluid particles will move exactly along these streamlines.  The above configuration therefore differs rather dramatically from that of a typical irrotational wave:  in the absence of vorticity and surface tension, it is know that \emph{none} of the particle paths can be closed \cite{constantin2007particle}.   }\label{streamlines figure}
\end{figure}
\end{centering}
We first  comment on the vorticity and the surface profile given in the above theorem. On the one hand, from Proposition \ref{prop:sign of Uy}, Corollary \ref{C:Unorms} and \eqref{Psi asymptotics}, we see that the kinetic energy is of $O(1)$. Roughly, 
\[
|v|_{L^2 (\Omega)} = |\nabla \Psi|_{L^2(\Omega)} = |\nabla U|_{L^2 (\R^2)} + o( e^{-\frac 1{2\delta}}),
\]
while the corresponding vorticity is spiked in the sense that
\[ 
\omega = \frac{1}{\delta^2} \gamma \left(U(\tfrac{\placeholder}{\delta} )\right) + o(e^{-\frac 1{2\delta}} ), \quad |\omega|_{L^\infty (\Omega)} =O\left(\tfrac 1{\delta^2}\right), \quad |\omega|_{L^1 (\Omega)} = |\Delta U|_{L^1 (\R^2)} + o( e^{-\frac 1{2\delta}}).
\]
On the other hand, the {total vorticity} is exponentially small in \(0 < \delta \ll 1\):
\[ 
\int_{\Omega} \omega \, \diff x = \int_\Omega \Delta \Psi \, \diff x = \int_{\p \Omega} N \cdot \nabla \Psi \,  \diff S = o( e^{-\frac 1{2\delta}}).
\]
By \eqref{Psi asymptotics} and the definition of $\Psi_0$, $\chi_\Omega (x) \omega(x) \to 0$ for any $x \ne 0$ as $\delta \searrow 0$, where $\chi_\Omega$ is the characteristic function of $\Omega$.  Then from the above integral estimate we can readily prove $\int_{\R^2} f \chi_\Omega \omega \, \diff x \to 0$ as $\delta \searrow 0$ for any continuous $f$ compactly supported in $\R^2$. Therefore, as a measure, $\chi_\Omega \omega\,  \diff x$ converges weakly to $0$ as $\delta \searrow 0$. 
 However, the vorticity has a rich spatial structure in a domain on the scale of $O(\delta)$ where its point-wise value is $O(\frac 1{\delta^{2}})$. Moreover, as $\omega$ is $O(1)$ in $ L^1(\Omega)$, these waves exhibit a highly localized but strong rotational vector field with kinetic energy of order $O(1)$.  

Since $\omega$ concentrates far away from $\p\Omega$ and the total vorticity is exponentially small, $\p \Omega$ is only weakly impacted by the spike.  This fact is reflected in the exponential smallness of $\eta$ in \eqref{eta asymptotics}. According to Proposition \ref{prop:sign of Uy}, 
the leading term $\eta_0$ given in \eqref{eta asymptotics} satisfies $- \eta_0 (x_1) \ge \frac 1C \delta^{-1/2} e^{-2/\delta}$ for $|x_1| < C^{-1}$ for some $C>0$ independent of $\delta$, while its tail is much smaller. Therefore the concentrated vorticity $\omega$ creates a surface depression in the near field with rapid decay as $|x_1| \to \infty$; see Figure~\ref{streamlines figure}.

\subsection{History and relation to our construction} \label{sec:intro:history}

Rotational steady water waves have been a very active area of research for nearly two decades, beginning with the construction of large-amplitude periodic gravity waves by Constantin and Strauss \cite{constantin2004exact}.  These authors used bifurcation theory starting from a fixed shear flow, and their methodology has since been adapted and expanded upon in many ways, see \cite{constantin2011book}.   It is important to note that, while there do exist explicit rotational water waves (for example, \cite{gerstner1809theorie,crapper1957exact,kinnersley1976exact,constantin2001edge,henry2008gerstner}), they are exceedingly rare.  From that perspective, the main contribution of \cite{constantin2004exact} was its systematic treatment of a broad class of vorticity distributions.  However, Constantin--Strauss  --- and most of the works that followed them --- require both that $\omega$ is non-localized and that there are no interior stagnation points. In particular, smooth perturbation of a shear flow could never yield 
decaying vorticity.  Interior stagnation and critical layers (regions of closed streamlines), however, can be constructed using variants of this approach. Early papers of Simmem and Saffman \cite{simmen1985steady} and da Silva and Peregrine \cite{dasilva1988steep} considered this regime through formal asymptotic analysis and numerical bifurcation theory.   In \cite{MR2409513}, it was rigorously shown that the nonlinear particle paths in the linearized system can have closed orbits, and the behavior of small waves with constant vorticity was studied.  Based on it, Wahl\'en \cite{wahlen2009steady} constructed exact periodic waves with one critical layer, and similar waves were subsequently constructed using a harmonic-functions approach, globally, in \cite{constantin2016global}. These works all treat constant vorticity and the situation where the linearized problem at the shear flow has a one-dimensional kernel. One can also find steady waves with critical layers bifurcating from two-dimensional \cite{ehrnstrom2011steady}, three-dimensional \cite{EW14ARMA}, and even arbitrarily high dimensional kernels \cite{MR3808597} of affine or near-affine vorticity functions, as well as from one-dimensional kernels of constant vorticity with one discontinuity \cite{MR3436246}.  Very recently, a global theory for analytic vorticity functions allowing for several critical layers has been presented in \cite{varholmthesis} (one might note that even affine vorticity can yield arbitrarily many vertically aligned stagnation points.)  The waves built in this paper have vorticity functions of the next order in this development, as Assumption~\ref{assumption:A} implies that $\gamma$ is nonlinear with leading-order linear term, although the method of proof is very different.

The first rigorous construction of traveling capillary-gravity waves with localized vorticity in infinite depth is due to Shatah, Walsh, and Zeng \cite{shatah2013travelling}.  In that paper, two classes of compactly supported vorticity were studied:  solitary and periodic waves with a submerged point vortex, and solitary waves with a vortex patch.  In the former case, $\omega$ is a Dirac measure supported in the interior of $\Omega$. This can be viewed as a solution to a suitably weakened version of the Euler equations.  The proof in \cite{shatah2013travelling} was based on a splitting of the velocity field into a rotational and irrotational component, followed by a bifurcation argument beginning at the trivial solution $(\Psi, \eta) = (0,0)$ with the total vorticity $\int \omega \, \diff x$ serving as the parameter. While the vortex patch solutions were small amplitude, the authors obtained a global curve of periodic traveling waves with a point vortex.  The vortex patches have finite energy and the corresponding vorticity is $C^{0,1}(\Omega)$ and smooth on its support.  
Later, Varholm \cite{varholm2016solitary} extended the ideas in \cite{shatah2013travelling} to the finite-depth case with arbitrarily many point vortices, and Le \cite{le2018existence} studied the existence and orbital stability of finite dipoles inside an infinite-depth capillary-gravity wave. Earlier work in the 50s and 60s that treated point vortices carried by gravity waves in finite depth include \cite{terkrkorov1958vortex,filippov1961motion,filippov1960vortex}. A vortex patch situated near a shoreline and such that the velocity vanishes completely outside a ball has also been constructed  in  \cite{constantin2011dynamical}, using dynamical systems tools.  
 
The capillary-gravity waves in the current work can be said to live between the above-mentioned types. They can be viewed, for \(0 < \delta \ll 1\), 
as smoothed vortex patches or as the limit, as the period tends to infinity, of steady periodic waves with critical-layers. 
We note that in 
\cite{shatah2013travelling}, (i) $\omega$ is single-signed and either a Dirac measure or in $C^{0,1}(\Omega)$;  and (ii)  the measure $\omega\, \diff x$ vanishes absolutely as one approaches the point of bifurcation.  By contrast, in the present paper,  the vorticity changes sign and is smooth \emph{throughout} $\Omega$.  Moreover,  $\omega \, \diff x$ converges to $0$ weakly as $\delta \searrow 0$, while the $L^1$ norm of $\omega$ and the kinetic energy both remain order $O(1)$.  This surprising feature results from 
 the fact that we do not perturb from a shear flow, but singularly from $U$ of \eqref{eq:ground state} which has fixed, positive energy. In all these respects, the vortex spikes constructed in Theorem~\ref{main euler theorem} contrast starkly with the literature described above.  

When $\omega$ is not compactly supported, it is of little help to decompose the velocity field into rotational and  irrotational parts. We are also barred from using shear flows as a model for the stream function. The main new idea is to instead look to the theory of 
spike and spike-layer solutions to singular perturbations of semi-linear elliptic PDE.  These equations typically have the form 
 \be 
 \delta^2 \Delta u  = u - u^p \qquad \textrm{in } D,\label{singularly perturbed PDE} 
 \ee
where $D \subset \mathbb{R}^n$ is a smooth bounded domain, $p > 1$, and Dirichlet or Neumann conditions are prescribed on $\partial D$.
Beginning in the late 80s, versions of \eqref{singularly perturbed PDE} were investigated intensively by the elliptic PDE community resulting in a vast literature;  see, for example, 
\cite{NiTa91, MR1639159, ni1995location, ni1998location}. 

Drawing inspiration from these works, we model our stream function as a rescaled and translated $U(\frac {\placeholder - x_*}\delta)$ on the unknown fluid domain represented by a conformal mapping $\Gamma$.
The translation invariance of the problem leads to a degeneracy --- as can be seen in Assumption~\ref{assumption:B} --- which is resolved through a Lyapunov--Schmidt reduction.  We outline heuristically how to solve the resulting highly degenerate bifurcation equation for $x_*$ in the next subsection. 

To the best of our knowledge, ours is the first work exploring singularly perturbed elliptic equations in the hydrodynamical context.  The method bears certain similarities to Li and Nirenberg's treatment of \eqref{singularly perturbed PDE} in \cite{MR1639159}, in particular, the use of a Lyapunov--Schmidt reduction, bundle coordinates in a tubular neighborhood of a family of translates, and boundary correction projections.  However, we stress that the steady water wave problem presents substantial new difficulties:  the upper boundary is free, the Bernoulli condition \eqref{eq:dynamic} imposed there is completely nonlinear, and the domain $\Omega$ is horizontally unbounded.

\subsection{Heuristic discussions
} \label{sec:intro:motivation}

In this subsection, we discuss several issues related to the {\it finite energy/spatial decaying} assumptions on {\it smooth} steady (stationary or traveling) solutions on fluid domains extending to horizontal infinity. We first observe that the support of the vorticity of such solutions should be the whole of $\Omega$. Otherwise, one expects that the  vorticity will not be smooth over the boundary of its support, as a consequence of the Hopf lemma for the elliptic equation \eqref{eq:main}. 

\subsubsection*{Traveling waves.} While we focus on stationary capillary-gravity waves in the current paper,  by shifting to a moving reference frame, Theorem~\ref{main euler theorem} immediately furnishes families of \emph{traveling} capillary-gravity waves with exponentially localized vorticity.  The velocity field for these waves will be an $H^{k_0-1}$ perturbation of a fixed uniform background current $c e_1 \neq 0$, and the vorticity will be spiked in the same sense as before. 

On the other hand, smooth \emph{finite}-energy waves with a non-zero wavespeed are unlikely to exist. In fact, the vorticity level curves for such waves would be closed loops $C_a = \{\omega =a\}$, which are transported by the velocity field $v = (v_1, v_2)$. Therefore 
\[
v \cdot \nu = c e_1 \cdot \nu = c \nu_1 \qquad \text{ along } \; C_a,
\] 
where $\nu = (\nu_1, \nu_2)$ is the unit outward normal vector of $C_a$. This implies that $|v| \ge \frac 12 |c|$ if $|\nu_1|>\frac 12$, which usually happens on an $O(1)$ proportion of most level curves. Consequently, $|v|$ is likely to be bounded from below on a set with infinite measure, which is prevented by the finite energy assumption.

\subsubsection*{Fluid depth and the boundary condition of the stream function.} In \cite{shatah2013travelling}, traveling capillary-gravity waves with compact vortex patches were constructed in fluids of infinite depth. Slightly modifying the formula of the rotational part of the velocity fields, actually the same construction should also work with finite depth. However, we do {\it not} expect smooth spatially localized stationary waves to exist in infinite depth unless the free surface is overturned. 

In fact, let us temporarily not preclude the possibility of $\Omega$ with infinite depth. Let a solution $\Psi$ of \eqref{eq:main} be given satisfying $v= \nabla^\perp \Psi  \in H^1 (\Omega)$ and $(v \cdot N)|_{\p \Omega} =0$ with $N = (N_1,N_2)$ the outward unit normal to $\Omega$. The latter condition implies that $\Psi$ is locally constant on $\p \Omega$. Fix $\Psi =0$ on $S = \graph{(1+ \eta)}$. 

Much as in the proof of Proposition \ref{prop:sign of Uy}, $\Psi$ and its derivatives decay exponentially as $|x| \to \infty$. Let $\Gamma$ be the antiderivative of $\gamma$ with $\Gamma(0) = 0$. We multiply \eqref{eq:main} by $\Psi_{x_2}$ and integrate to find
\be \label{E:motivation-1} \begin{split} 
\frac{1}{\delta^2} \int_{\p \Omega} \Gamma(\Psi) N_2 \, \diff S & = \frac{1}{\delta^2} \int_{\Omega} \partial_{x_2} \Gamma(\Psi) \, \diff x  = \int_{\Omega}  \p_{x_2} \Psi \Delta \Psi \, \diff x \\
&= - \int_{\Omega} \nabla \p_{x_2} \Psi \cdot \nabla \Psi \, \diff x + \int_{\partial \Omega} \p_{x_2} \Psi \nabla \Psi \cdot N\, \diff S \\   
&= -\frac{1}{2} \int_{\partial\Omega} | \nabla \Psi|^2 N_2 \, \diff S + \int_{\partial \Omega} \p_{x_2} \Psi N \cdot \nabla \Psi \, \diff S \\
& = \frac{1}{2} \int_{\partial\Omega} |\nabla \Psi|^2 N_2 \, \diff S, 
\end{split} \ee
where in the last step above we used that $\Psi$ is locally constant on $\partial \Omega$ and thus \( \nabla \Psi = (N \cdot  \nabla \Psi) N\) holds there. 

The first implication of this equality is that if $\Psi$ is nontrivial, then $S  \subsetneq \p \Omega$.
Otherwise we would have $\int_{S} |\nabla \Psi|^2 N_2 \, \diff S = 0$ with $N_2 >0$, which is impossible. This argument does not rely on anything but the regularity of $\gamma$, in particular, we do not need the full strength of Assumption~\ref{assumption:A} or~\ref{assumption:B}. Non-existence of deep water solitary waves in the presence of algebraically localized vorticity has been more thoroughly investigated in the recent paper  \cite{chen2019existence}; see also \cite{craig2002nonexistence,sun1997analytical,hur2012no,wheeler2018integral} for results on the irrotational case.

Now suppose instead that the domain is finite depth, and set $\p \Omega = S \cup B$, with $B = \{x_2 =-1\}$ denoting the flat rigid lower boundary. Suppose also that $S \cap B =\emptyset$. The properties (i) $|\eta(x_1)| \to 0$ as $|x_1| \to \infty$; (ii) $\nabla \Psi \in L^2 (\Omega)$; and (iii) $\Psi$ is locally constant on $\p \Omega$, together imply that $\Psi|_B = \Psi|_S=0$ based on a simple H\"older estimate on $\Psi$ along vertical lines. Therefore, from \eqref{E:motivation-1}, we infer that 
\be   
\frac{1}{2} \int_{\partial\Omega} |\nabla \Psi|^2 N_2 \, \diff S = 0. \label{motivational identity} 
\ee

\subsubsection*{The reduced (degenerate) equation from the Lyapunov--Schmidt reduction.} 
Equation \eqref{motivational identity} is the key to the proof of our main theorem. 
As mentioned above, we first carry out a Lyapunov--Schmidt reduction argument to reduce the problem to a highly degenerate one-dimensional ``bifurcation'' equation with the parameter $\tau$ as in Theorem \ref{main euler theorem}. One of the usual techniques to handle those somewhat degenerate bifurcation equations is to first use a blow-up argument to search for a non-degenerate direction of the linearized problem, then employ the implicit function theorem. Even though Proposition~\ref{prop:U_0} does imply such linear invertibility of the bifurcation equation, the non-degeneracy we find is far too weak for an (obvious) application of the implicit function theorem to be effective.   

Instead, in Section \ref{proof main theorem sec} we show that the bifurcation equation is equivalent to the above \eqref{motivational identity}.  Now, on the free surface, $N_2 = \left(1+ (\eta')^2 \right)^{-\frac 12}
> 0$, while  $N_2 = -1$ on the flat bed $B$.   If $\nabla \Psi$ is highly localized close to the surface, then the integral there should dominate so that the left-hand side of \eqref{motivational identity} would be positive, and conversely for $\nabla\Psi$ concentrated near the bed.  This mandates a \emph{balancing} between the contributions on the surface and bed. 
That observation is at the heart of the analysis in the last part of Section~\ref{proof main theorem sec}.  It also reveals the importance of the translation parameter $\tau$.  

\subsubsection*{Non-flat bottom and more.} With some modifications, the approach of the current paper should also apply when the bed has nontrivial topography\footnote{This question was also raised by Shuangjie Peng and Shusen Yan during a talk given by the third author.}. Indeed, suppose $\p \Omega = S \cup B$, where $B$ is now a horizontally asymptotically flat rigid bottom, for simplicity taken even in $x_1$. Thus we expect the vorticity to be localized at $\tau e_2 = (0, \tau) \in \Omega$ for some $\tau$. We can parametrize the unknown $\Omega$ by a conformal mapping defined on a fixed domain above $B$ and below $\{x_2=1\}$. Based on Proposition~\ref{prop:sign of Uy}, one may adjust the basic estimates in Sections~\ref{estimates section} and \ref{sec:Spec} accordingly to carry out the Lyapunov--Schmidt reduction and arrive at a highly degenerate one-dimensional reduced bifurcation equation that would still turn out to be equivalent to \eqref{motivational identity}. As in the current paper, the distance from $\tau e_2$ would again play a crucial role. Let 
\[
d (\tau) = \dist (\tau e_2, \p \Omega). 
\]
Much like \cite{MR1639159}, stationary solutions  are expected to exist with  a localized vorticity concentrated near strict local maximums of $d(\tau)$. However, when multiple localized vorticity locations are considered or when $B$ is not necessarily even in $x_1$, a sphere packing problem arises. See, for example, \cite{GW99}.

Lastly, we remark that it would be very interesting, though quite difficult, to study gravity water waves with a spike vortex.  Surface tension allows us to treat the Bernoulli condition \eqref{eq:dynamic} essentially as an elliptic problem on the boundary.  In fact, the linear part is invertible, which greatly simplifies the analysis; see the proof of Lemma~\ref{reduction lemma}.   Perhaps with much more careful estimates it would be possible to allow for $\sigma = 0$. 

\subsection{Plan} \label{sec:intro:plan} 

We begin, in Section~\ref{sec:perturbation}, by rewriting the stationary water wave problem into an analytically more  tractable form. Using a conformal mapping $\Gamma$,  the fluid domain is pulled back to a fixed slab $\slab_\delta$ of width $2/\delta$; this mapping $\Gamma$ becomes one of the unknowns, taking the place of $\eta$.  We impose the desired ansatz \eqref{eq:def psi} on the stream function, thereby reformulating the problem in terms of the deviation of $\Psi$ from a translated and rescaled solution $U$ to \eqref{eq:ground state}.

In Section~\ref{estimates section}, we obtain  leading-order approximations of $U$ and a boundary correction operator as well as rather precise exponentially small bounds on the remainders.

Section~\ref{sec:Spec} is devoted to the study of the linearized problem 
at an approximate solution. 
Specifically, we prove that there is a small simple eigenvalue $l = l(\delta) = O(e^{\frac {2-|\tau|}\delta})$ related to the direction of $\p_{x_2} U$. The linearized problem is uniformly non-degenerate in the complementary codimension-1 directions. 

All of these tools are used in Section~\ref{proof main theorem sec} to prove Theorem~\ref{main euler theorem}.  Adopting bundle-type coordinates over $\tau \in (-\frac 13, \frac 13)$, 
we carry out a Lyapunov--Schmidt reduction in the non-degenerate codimension-1 directions to reduce the problem to 
a one-dimensional bifurcation equation. As mentioned above, this bifurcation equation is equivalent to \eqref{motivational identity} and the proof 
is completed by invoking the intermediate value theorem.  Here the idea of balancing the two surface integrals is made rigorous through careful estimates of all the quantities involved.  Indeed, while this analysis is quite delicate, the simple identity \eqref{motivational identity} is the key to the argument.    

\subsection*{Notation.}  Throughout the paper  \(\lesssim\), \(\gtrsim\) and \(\eqsim\) indicate relations that are valid up to a positive factor which can be chosen uniformly in \(\delta\) small enough and $\tau \in [-\frac 13, \frac 13]$. 
Complex scalars are sometimes viewed as 2-d real vectors, hence ``$\cdot$" between complex quantities denotes their dot product. For a given $L^2$ function $f \ne 0$ defined on certain domain, we often use $f^\perp$ to denote the $L^2$-orthogonal complement of $f$. We also use $\Diff = \partial_{x_1}$.


\section{Reformulation}\label{sec:perturbation}

As the first step toward proving Theorem~\ref{main euler theorem}, the stationary water wave problem \eqref{eq:main}, \eqref{E:SKinematic}, and \eqref{eq:dynamic} will be reformulated on a fixed domain, and we will build in the spike ansatz for the stream function mentioned in \eqref{eq:def psi}. The final product of these efforts is an equivalent transformed problem \eqref{perturbed problem} that is posed on an infinite strip.

\subsection{Rescaling and parametrization}

We start by introducing new coordinates that eliminate the free boundary.  As mentioned in the introduction, this can be achieved at little cost in the irrotational regime; see, for example, \cite{constantin2009pressure,constantin2010particle}.  With vorticity, however, one expect to pay a price in the form of increased complexity of the equations.   Given that the highest-order operator in the semi-linear equation \eqref{eq:main} is the Laplacian, it is natural to work with conformal mappings.  With that in mind, define the \emph{reference domain} to be \(\R \times (-1,1)\), which we identify with the complex strip
\[
\C_{|z_2| < 1}  = \{ z  = z_1 + \I z_2 \in \C \colon  |z_2| < 1 \}.
\]
We will look for fluid domains $\Omega$ that are expressed as the image of the reference domain under a near-identity holomorphic mapping.  Specifically, let $\Gamma = \Gamma_1 + \I \Gamma_2 \colon \C_{|z_2| < 1} \to \C$  be holomorphic and satisfy 
\be \label{E:conformal}
|\Gamma|_{H^{\frac{5}{2}}} \ll 1, \qquad \Gamma(-\overline z) = - \overline{\Gamma(z)}, \qquad {\Gamma_2|_{\{ z_2 = -1\}} = 0}. 
\ee
Note this implies that \(z_1 \mapsto \Gamma_1\) is odd and \(z_1 \mapsto \Gamma_2\) is even. 
Such a conformal mapping is uniquely determined by $\Gamma_2|_{x_2 =1}$. In fact, since \(\Gamma_2\) is harmonic, 
\[
(\partial_{x_2}^2 - \xi_1^2) \F_{x_1} \Gamma_2 = 0,
\]
where $\F_{x_1}$ denotes the Fourier transform in the $x_1$ variable. By construction, $\Gamma_2$ vanishes on the bottom of the domain, and hence it depends analytically upon its trace on the upper boundary, \(\Gamma_2(\placeholder,1)\). Explicitly, 
\[ 
\left(\F_{x_1} \Gamma_2\right) (\xi_1, x_2) = \frac{\sinh{\left( |\xi_1|(x_2+1) \right)}}{\sinh{\left(2|\xi_1|\right)} }  (\F_{x_1} \Gamma_2) (\xi_1,1) \qquad \textrm{for all } \xi_1 \in \mathbb{R}, \, |x_2| < 1,
\]
so we have
\be \label{E:HarmExt} 
\Gamma_2 = \frac{\sinh{\left( |\partial_{x_1}|(x_2+1) \right)}}{\sinh{\left(2|\partial_{x_1}|\right)} }   \Gamma_2(\placeholder,1) \qquad \textrm{in } {\C}_{|z_2| < 1},
\ee 
and
\be \label{E:DN}
\partial_{x_2} \Gamma_2  = |\partial_{x_1}| \coth{\left(2|\partial_{x_1}| \right)} \Gamma_2  \qquad \text{on } \{x_2 = 1\}.
\ee
Observe that $|\partial_{x_1}| \coth{\left(2|\partial_{x_1}| \right)}$ above is the Dirichlet--Neumann operator on the strip $\{|x_2| < 1\}$ with a homogeneous Dirichlet condition imposed on the lower boundary. The real part $\Gamma_1$ is a harmonic conjugate of $\Gamma_2$ whose one degree of freedom is fixed by the symmetry in $x_1$. 

The corresponding fluid domain is taken to be 
\[
\Omega := (\id+\Gamma)(\{|x_2| < 1\}) = \left\{ \left(x_1 + \Gamma_1(x), \, x_2 + \Gamma_2(x) \right) \colon |x_2| < 1 \right\}. 
\]
It follows that the free surface is parameterized by $x_1 \mapsto (x_1, 1) + \Gamma(x_1, 1)$, for $x_1$ ranging over $\mathbb{R}$.  This curve can also be written as the graph 
\be \label{E:eta}
x_2 =1+ \eta (x_1), \qquad  \eta = \Gamma_2 \circ \left( \id + \Gamma_1(\placeholder , 1) \right)^{-1}. 
\ee

The stream function can be pulled back,
\be \label{relation Phi to Psi}
\Phi = \Psi \circ (\id+\Gamma) \colon \{|x_2| < 1\} \to \R,
\ee
yielding a new unknown defined on the fixed reference domain.   
Then the water wave problem  \eqref{eq:main}, \eqref{E:SKinematic}, and \eqref{eq:dynamic}  become
\[
\left \{ \begin{aligned}
\delta^2 \Delta \Phi &=   |1 + \Gamma^\prime|^2 \,  \gamma( \Phi ) & \qquad\text{ in } \{|x_2| < 1\}, \\
\Phi  &= 0 & \qquad\text{ on } \{|x_2| = 1\}, 
\end{aligned} \right.
\]
together with the transformed Bernoulli condition 
\begin{equation}
\frac{1}{2}  \frac{(\p_{x_2} \Phi )^2}{|1 +\Gamma^\prime|^{2}}   - \alpha^2  \frac{\im (\Gamma^{\prime\prime} (1 + \overline{\Gamma^\prime}))}{| 1 + \Gamma^\prime|^{3}} + g \Gamma_2 = 0 \qquad \text{ on }  \{x_2 = 1\}, \label{eq:dynamic_Phi}
\end{equation}
where we used that $|\nabla \Phi| = |\p_{x_2} \Phi|$ along $x_2 =1$ due to the boundary condition on $\Phi$. Here \(\Gamma^\prime = \partial_z \Gamma = \partial_{x_1} \Gamma\) in view of the fact that \(\Gamma\) is holomorphic, and 
 \[ 
 \kappa = -   \frac{\im (\Gamma^{\prime\prime} (1 + \overline{\Gamma^\prime})) }{| 1 + \Gamma^\prime|^{3}}  \]
 is the signed curvature of the interface. 
 
 Recalling the scaling in \eqref{eq:def psi}, we define 
\be \label{relation phi and Phi}
\varphi  = \Phi (\delta \placeholder),
\ee
which then solves a non-dimensionalized version of the problem for $\Phi$ set on the slab 
\be \slab_\delta = \left\{ x \in \mathbb{R}^2 : |x_2| < \tfrac{1}{\delta} \right\}. \label{def slab} \ee
It is important to realize that this domain is decreasing in \(\delta\), so that in particular \(\slab_{2\delta} \subset \slab_\delta\).  Now it is easy to compute that $\varphi$ satisfies the elliptic equation
\be \label{eq:varphi elliptic} \left\{ \begin{aligned}
\Delta \varphi &=   |1 + \Gamma^\prime(\delta \placeholder)|^2 \,  \gamma( \varphi ) && \qquad\text{in } \slab_\delta\\
\varphi &= 0 && \qquad \textrm{on } \partial\slab_\delta, 
\end{aligned} \right.\ee
and the Bernoulli condition translates to 
\begin{equation}
\frac{1}{2 \delta^{2}}  \frac{\left(\p_{x_2} \varphi(\placeholder,\tfrac{1}{\delta}) \right)^2}{|1 +\Gamma^\prime(\delta \placeholder ,1)|^{2}}   - \alpha^2  \frac{\im (\Gamma^{\prime\prime}(\delta \placeholder, 1) (1 + \overline{\Gamma^\prime(\delta \placeholder, 1)}))}{| 1 + \Gamma^\prime(\delta \placeholder,1)|^{3}} + g \Gamma_2(\delta \placeholder,1) = 0.\label{eq:dynamic_varphi}
\end{equation}

\subsection{Boundary correction} 

Our overarching strategy is to model $\Psi$, and by extension $\varphi$, on a rescaled $U$ of \eqref{eq:ground state}.  However, while $U$ is exponentially localized, it does not satisfy the homogeneous boundary conditions in \eqref{eq:varphi elliptic}.  We  therefore perform a \emph{boundary correction}, modeled on Assumption~\ref{assumption:A}, subtracting a function from $U$ that shares its trace but is exceedingly smaller in the interior. 

For any real-valued function \(f\) in a reasonable Sobolev space (see below) defined on 
$\partial \slab_\delta$, we introduce the extension operator
\[
\mathrm{bc} \colon f \mapsto f_\mathrm{bc},
\]
defined uniquely by
\begin{equation}\label{bc fourier formula} 
 (\F_{x_1} f_{\mathrm{bc}})(\xi_1, x_2)  =  \sum_{\pm} \pm \frac{\sinh{\left( \jbracket{\xi_1}(x_2 \pm \frac 1\delta) \right)} (\F_{x_1} f_{\pm})(\xi_1)}{\sinh{\left(\frac{2\jbracket{\xi_1}}{\delta}\right)}},  
 \end{equation} 
where \(f_{\pm}\) is the restriction of \(f\) on \(\{x_2 = \pm \tfrac{1}{\delta}\}\), and we are using the Japanese bracket notation $\jbracket{\xi_1} = (1+|\xi_1|^2)^{1/2}$.
Provided  that \(f_\pm \in H^s (\R)\), $s>0$, the function \( f_{\mathrm{bc}}\) is an \(H^{s+\frac12} (\slab_\delta)\)-solution\footnote{In general the solution of \eqref{definition bc} need not be unique, as \(\slab_\delta\) is an infinite slab.} of 
\be \label{definition bc} \left\{
\begin{split}
(1 - \Delta) f_\mathrm{bc} &= 0 \qquad \textrm{in } \slab_{\delta} ,\\
f_\mathrm{bc} &= f \qquad \textrm{on } \partial \slab_\delta.
\end{split} \right. \ee
Observe that, due to the localization of $U$ and the assumption $\gamma'(0)=1$, the above problem closely resembles the linearized operator $\gamma^\prime(U) - \Delta$ away from the origin.  
It is worth noting that, for any $s>0$, 
\[
|{ f_{\mathrm{bc}}}|_{H^{s+\frac 12} (\slab_\delta)} \lesssim 
|f|_{H^s (\p \slab_\delta)}.
\]

\subsection{The perturbed problem} \label{perturbed problem section}

Let $U$ be given by Assumption~\ref{assumption:A}.  As discussed in Section~\ref{sec:formulation}, it will be important to consider vertical translates of this function.  For each $\tau \in [-\tfrac{1}{3}, \tfrac{1}{3}]$, let
\begin{equation}\label{eq:U_tau}
U(\placeholder, \tau) = U(\placeholder -\tfrac{\tau}{\delta} e_2).
\end{equation}
With a slight abuse of notation we shall still write \(U\) to denote the function \(U(\placeholder, \tau)\), except when the precise value of \(\tau\) becomes important. At other times, it will be more convenient to use the notation $U(\tau)(\placeholder)$ rather than $U(\placeholder, \tau)$.  
The value \(\frac{1}{3}\) is  unimportant; we will find waves for \(|\tau|\) exceedingly much smaller.
What is important is that the center of vorticity remains closer to the origin than to the boundary of the reference domain, but $\frac 13$ has no special significance.  

We proceed with the \emph{ansatz}
\be \label{definition u}
\varphi {=}  u + U - U_\mathrm{bc},
\ee
where $\textrm{bc}$ denotes the boundary correction from \eqref{definition bc}. Thus $u$ measures the deviation of $\varphi$ from the rescaled, translated, and boundary corrected $U$.  Inserting this into \eqref{eq:varphi elliptic}, we see that it solves the following elliptic PDE set on $\slab_\delta$:
\begin{equation}\label{eq:main_u}\\
\begin{aligned}
\Delta u &= | 1 + \Gamma^\prime(\delta \placeholder) |^2 \gamma(u + U - U_\mathrm{bc}) - \gamma(U)  + U_\mathrm{bc}\\
&= \gamma^\prime(U) u +  | 1 + \Gamma^\prime(\delta \placeholder) |^2 \gamma(u + U - U_\mathrm{bc}) - \gamma(U)  - \gamma^\prime(U) u + U_\mathrm{bc}.
\end{aligned}
\end{equation}
Here, we have made use of the facts that \(\Delta U = \gamma(U)\) and \( \Delta U_\mathrm{bc} = U_\mathrm{bc}\). Similarly, the kinematic condition in \eqref{eq:varphi elliptic} takes the form
\begin{equation}\label{eq:kinematic_u}
u = 0 \qquad \textrm{on } \partial\slab_\delta,
\end{equation}
since $U = U_\bc$ there.  

Consider next the Bernoulli condition \eqref{eq:dynamic_varphi}.  Direct substitution yields 
\be
\begin{split}
0 & = \frac{1}{2 \delta^{2}}  \frac{\left(\p_{x_2} (u + U - U_\mathrm{bc})(\placeholder,\tfrac{1}{\delta}) \right)^2}{|1 +\Gamma^\prime(\delta \placeholder,1)|^{2}}   - \alpha^2  \frac{\im (\Gamma^{\prime\prime}(\delta \placeholder, 1) (1 + \overline{\Gamma^\prime(\delta \placeholder, 1)}))}{| 1 + \Gamma^\prime(\delta \placeholder,1)|^{3}} \\ 
& \qquad + g \Gamma_2(\delta \placeholder,1). 
\end{split}
\label{eq:dynamic_u}
\ee
From the Cauchy--Riemann equations and 
\[
\Gamma^\prime = \partial_{x_1} \Gamma = \partial_{x_2} \Gamma_2 + \I \partial_{x_1} \Gamma_2,
\]
any derivatives involving \(\Gamma\) can be expressed in terms of derivatives of \(\Gamma_2\). Making this replacement in \eqref{eq:dynamic_u} yields
\begin{equation}
\begin{split} 
0  = & 
\frac{1}{2 \delta^{2}}  \frac{\left( \p_{x_2} (u + U - U_\mathrm{bc})\left(\frac{\placeholder}{\delta},\frac{1}{\delta}\right) \right)^2}{(1 +\partial_{x_2}\Gamma_2)^2 + (\partial_{x_1}\Gamma_2)^2}   - \alpha^2  \frac{(1 + \partial_{x_2}\Gamma_2) \partial_{x_1}^2 \Gamma_2 - \partial_{x_1} \Gamma_2 \partial_{x_1x_2} \Gamma_2 }{((1 +\partial_{x_2}\Gamma_2)^2 + (\partial_{x_1}\Gamma_2)^2)^{3/2}}  + g \Gamma_2. \end{split} \label{eq:dynamic_Gamma2}
\end{equation}
Here, all terms involving \(\Gamma_2\) are evaluated at \((x_1,1)\). The idea is that, to the leading order in terms of $\Gamma_2$, the right-hand side of \eqref{eq:dynamic_Gamma2} is determined by the operator $g - \alpha^2 \partial_{x_1}^2$ acting on $\Gamma_2$, which is invertible $H^s(\mathbb{R}) \to H^{s-2}(\mathbb{R})$ for all $s \in \R$. To make this rigorous, 
let 
\be \label{def Gammas}
\Gammas  = \Gamma_2(\placeholder,1)
\ee
be the trace of $\Gamma_2$ on the top of the reference domain $\slab_\delta$.
We have from \eqref{eq:dynamic_Gamma2} and \eqref{E:DN} 
\begin{equation}
\frac{1}{2 \delta^{2}}  \frac{\left(\p_{x_2} (u + U - U_\mathrm{bc})\left( \frac{\placeholder}{\delta},\frac{1}{\delta} \right) \right)^2}{\left(1 +m(\Diff)\Gammas\right)^2 + {\Gammas^\prime}^2}   - \alpha^2  \frac{ (1 + m(\Diff) \Gammas) \Gammas^{\prime\prime}  - \Gammas^\prime m(\Diff) \Gammas^\prime   }{\left((1 +m(\Diff)\Gammas)^2 + {\Gammas^\prime}^2\right)^{3/2}} + g \Gammas = 0, \label{eq:dynamic_eta}
\end{equation}
where \(\Diff = \partial_{x_1}\), $m(\Diff) = |\Diff| \coth{(2|\Diff|)}$,  and \(\Gammas^\prime = \partial_{x_1} \Gammas\). Let $A(\Gammas)$ be a linear operator depending on $\Gammas$ acting on $v\colon \R \to \R$ as   
\[
A(\Gammas) := \left(g- \alpha^2 \left((1 +m(\Diff)\Gammas)^2 + {\Gammas^\prime}^2 \right)^{-\frac{3}{2}} \left(\left(1 + m(\Diff)\Gammas\right) \Diff^2  - \Gammas^\prime m (\Diff) \Diff \right) \right) (g-\alpha^2 \Diff^2)^{-1}.  
\]
Notice also that $|\Diff|$ preserves the even-odd parity. 
For a given smooth \(\Gammas\), \(A\) is a zero-order operator on any Sobolev space \(H_{\mathrm{e}}^s(\R)\), \(s \in \R\),  where here and elsewhere the subscript `e' indicates that the functions are even in $x_1$.
More precisely, if \(\Gammas \in H_{\mathrm{e}}^s(\R)\) for \(s > 3/2\), then the map\footnote{Note here that the lower right \(\mathrm{s}\) used in $\Gammas$ stands for `surface', while the slanted \(s\) is a (general) regularity index.}
\begin{equation}\label{eq:A-regularity}
H_{\mathrm{e}}^s(\mathbb{R}) \ni \Gammas \mapsto A(\Gammas) \in \mathcal{L}(H_{\mathrm{e}}^{s'}(\mathbb{R})) \text{ is analytic, } s' \in [1-s, s-1].
\end{equation}
Recall here that \(H^{-s}(\R)\) is the continuous dual of \(H^{s}(\R)\), whence the lower bound \(1-s\) is needed to ensure that products can be made sense of when applying \(A(\Gammas)\) to \(H^{s'}(\R)\). 
Now \(A(0) = \id\) and we have the bound
\[
| A(\Gammas) - \id|_{\mathcal L(H^{s^\prime})} \lesssim |\Gammas|_{H^{s}}.
\]
Thus \(A(\Gammas) \in \mathcal{L}(H^{s'})$ is invertible for \( |\Gammas|_{H^{s}} \ll 1\). We can now isolate the leading-order terms in \eqref{eq:dynamic_eta} by applying \(A(\Gammas)^{-1}\) to it:
\begin{equation}
\begin{aligned}
0 =  (g   - \alpha^2  \Diff^2) \Gammas 
+  \frac{1}{2 \delta^{2}} A(\Gammas)^{-1} \left[ \frac{\left(\p_{x_2} (u + U - U_\mathrm{bc})\left( \frac{\placeholder}{\delta},\frac{1}{\delta}\right) \right)^2}{(1 +m(\Diff)\Gammas)^2 + {\Gammas^\prime}^2}    \right],
\end{aligned} \label{eq:dynamic_inverse}
\end{equation}
which is valid as long as \( |\Gammas|_{H^{s}} \ll 1\).

Now we are roughly in a position to make a rigorous statement of our problem.  For any $\delta > 0$, we define 
\be \label{E:space} 
X_\delta^k  = H_{\textrm{e}}^k (\slab_\delta) \cap H_0^1(\slab_\delta), \quad k\ge1,\qquad  X_\delta^0 = L_{\textrm{e}}^2(\slab_\delta).  
\ee
Summarizing the analysis of this section, we see that if \(u\), \(\Gammas\), and \(\tau\) satisfy 
\bse \label{perturbed problem}
\begin{align}
(-\Delta+\gamma^\prime(U)) u + F(\tau,u, \Gammas) & = 0 \qquad \textrm{in } \slab_\delta \label{perturbed u eq} \\
(g-\alpha^2 \Diff^2) \Gammas + G(\tau, u, \Gammas) & = 0 \qquad \textrm{on } \mathbb{R}, \label{perturbed Gamma eq} 
\end{align}
\ese
where $F\colon [-\tfrac{1}{3}, \tfrac{1}{3}] \times  X_\delta^k \times H_{\mathrm{e}}^{k}(\mathbb{R}) \to X_\delta^{k-2}$ and $G\colon [-\tfrac{1}{3}, \tfrac{1}{3}] \times X_\delta^k \times H_{\mathrm{e}}^{k}(\mathbb{R}) \to H^{k-2}_{\text{e}}(\R)$ are the mappings
\be 
F(\tau, \placeholder)\colon (u, \Gammas) \mapsto |1+\Gamma^\prime(\delta \placeholder )|^2 \gamma(u+U-U_{\textrm{bc}}) - \gamma(U) - \gamma^\prime(U)u + U_{\textrm{bc}},  \label{def F} \ee
and
\be \begin{split}
G(\tau, \placeholder)\colon (u, \Gammas) \mapsto  
frac{1}{2 \delta^{2}} A(\Gammas)^{-1} \left[ \frac{\left(\p_{x_2} (u + U - U_\mathrm{bc})(\frac{\placeholder}{\delta},\frac{1}{\delta}) \right)^2}{(1 +|\Diff| \coth{(2|\Diff|)} \Gammas)^2 + {\Gammas^\prime}^2}    \right], 
\end{split} \label{def G} \ee
then $(\Psi, \eta)$, reconstructed via \eqref{relation Phi to Psi}, \eqref{relation phi and Phi}, \eqref{definition u}, \eqref{E:conformal}, \eqref{def Gammas}, and the Cauchy-Riemann equations, will solve the stationary water wave problem  \eqref{eq:main}, \eqref{E:SKinematic}, and \eqref{eq:dynamic}.  Recall here that \(U\) is a shorthand for \(U(\placeholder, \tau) = U(\placeholder -\tfrac{\tau}{\delta} e_2)\). For the class of \(\gamma\) satisfying Assumption~\ref{assumption:A} and Assumption~\ref{assumption:B}, the mappings $F$ and $G$ are well defined and continuously differentiable given some basic estimates on $U$ and $U_{\textrm{bc}}$ that are derived in the next section. For that reason, we postpone making a precise statement, or offering a proof, until Lemma~\ref{equivalence lemma}.

\section{Estimates of $U$ and its boundary corrections} \label{estimates section}

This section is devoted to the estimates of $U(\placeholder, \tau)$, its derivatives, and boundary corrections of the same functions, assuming that Assumptions~\ref{assumption:A} and~\ref{assumption:B} from Section~\ref{sec:formulation} hold. 
Finally, we give some estimates of the  nonlinearities $F$ and $G$, defined in \eqref{def F} and \eqref{def G}, in the elliptic system \eqref{perturbed problem}, which is equivalent to the original problem \eqref{eq:main}, \eqref{E:SKinematic}, and \eqref{eq:dynamic} of finding stationary water waves. 
\subsection*{Estimates for $U$, \((\partial_{x_2} U)_\mathrm{bc}\), and  \(\partial_{x_2} (U_\mathrm{bc})\) 
}
We start with some basic estimates on $U$. Recall from Assumption~\ref{assumption:A} that $\gamma \in C^{k_0}$, for a fixed integer $k_0\ge 2$.  

\begin{proposition} \label{prop:sign of Uy} 
Under Assumptions~\ref{assumption:A} and~\ref{assumption:B}, there exists $\lambda \ne 0$ such that 
 \begin{equation}\label{eq:rate of decay}
\lim_{r \to \infty} (-1)^{k} r^{\frac{1}{2}}\e^{r} \p_r^k U(r) = \lambda, \qquad\text{  for all } \; 0\le k \le k_0 +2.
 \end{equation}
\end{proposition}

\begin{remark}\label{rem:rate of decay}
As 
\[
\nabla = \begin{pmatrix} \cos \theta & -\frac 1r \sin \theta \\ \sin \theta & \frac 1r \cos \theta \end{pmatrix} \begin{pmatrix} \p_r \\ \p_\theta \end{pmatrix},
\]
and 
\[ 
\p_{x_2}^2 = \sin^2 \theta \p_r^2 + \frac {2\cos\theta \sin \theta}r \p_{r \theta} + \frac {\cos^2 \theta}{r^2} \p_\theta^2 + \frac {\cos^2\theta}r \p_r -\frac{\cos\theta \sin \theta}{r^2}\p_\theta,  
\] 
when applied to radial functions, we have 
\be \label{E:p2U}
\p_{x_2} U= \sin \theta U_r, \quad \p_{x_2}^2 U = \sin^2 \theta U_{rr} + \frac {\cos^2\theta}r U_r.
\ee
Proposition~\ref{prop:sign of Uy} readily induces signs on the Cartesian derivatives of these functions. In particular, \(\sign \partial_{x_2} U = -\sign (\lambda x_2)\) globally  with
\[
\partial_{x_2} U \eqsim  - \lambda x_2 r^{-\frac{3}{2}} e^{-r}, \quad \p_{x_2}^2 U \eqsim \lambda x_2^2 r^{-\frac{5}{2}} e^{-r}
\]
when \(r 
\gg 1\). Note also that
\[
|1 - \gamma'(U)| \lesssim r^{-\frac{1}{2}} e^{-r}, \quad |\gamma(U) - U| \lesssim r^{-1} e^{-2r}, \qquad r\gg 1.
\]
\end{remark}

\begin{remark}\label{rem:sign of Uy tau}
For \(|\tau|< \frac{1}{3}\), the function \(U(\placeholder, \tau)\) from \eqref{eq:U_tau} is just a translation of the center and global maximum of the radial function \(U\) from the origin to \((0,\frac{\tau}{\delta})\). It follows that Proposition~\ref{prop:sign of Uy} applies to 
\(U(\placeholder, \tau)\) 
with \(r\) changed accordingly.
\end{remark}

\begin{remark} \label{R:UDecay}
The solution $U$ to \eqref{eq:ground state} is often obtained through a variational approach carried out in $H^1$ space. In fact, for any $\gamma \in C^1 (\R^2)$ satisfying $\gamma(0)=0$, any radial solution $U \in C^{2} (\R^2) \cap H^1 (\R^2)$ automatically satisfies the decay assumption $\lim_{r\to \infty}  U(r) = \lim_{r\to \infty} U'(r) =0$ in \ref{assumption:A}. This is due to the inequality 
\[
|U(r_2) - U(r_1)| = |\int_{r_1}^{r_2} U'(r) dr| \le |U'|_{L^2 (\R^2)} |\int_{r_1}^{r_2} \frac 1r dr|^{\frac 12} =  |U'|_{L^2 (\R^2)} |\log \frac {r_2}{r_1} |^{\frac 12}
\] 
which implies $ \lim_{r\to \infty} U(r) =0$ hence $U$ also bounded. The boundedness of $U$ yields $\Delta U = \gamma(U) \in L^2(\R^2)$ and thus $U \in H^2 (\R^2)$. Therefore $ \lim_{r\to \infty} U'(r) =0$ follows from the same argument. Obviously one may also replace $U \in C^{2} (\R^2) \cap H^1 (\R^2)$ by $U \in C^{2} (\R^2) \cap L^2 (\R^2) \cap L^\infty (\R^2)$.  
\end{remark}

\begin{proof}[Proof of Proposition~\ref{prop:sign of Uy}]  
The decay rate \eqref{eq:rate of decay} is stated by Li and Nirenberg \cite{MR1639159} for the case \(\gamma(t) = t - t^p\), with a reference to an earlier paper of Berestycki and P.~L. Lions  \cite{MR695535}.  However, while that work could be extended to our setting, as written it does not contain the same sharp result and it is restricted to three or higher dimensions. Here we provide a sketch of a proof that does cover the case of interest; it is based on invariant manifold methods rather than variational techniques. For a reference, see for example \cite{MR2224508}.

In polar coordinates the semi-linear problem for $U$ is
\begin{equation}\label{eq:radial coordinates}
\partial_r^2 U = -\frac{1}{r} \partial_r U + \gamma(U)
\end{equation}
and thus it suffices to obtain the estimate for $k=0, 1$, due to the fact $\gamma'(0)=1$. Letting
\[
w_1 = \frac 12(U + \p_r U), \quad w_2 = \frac 12 (U- \p_r U), \quad s = \frac 1r, \quad \gamma_1 (U) = \gamma(U) - U = O(U^2), 
\]
we rewrite \eqref{eq:radial coordinates} as 
\be \label{E:blow-up} \begin{cases}
\p_r w_1 = (1-\frac s 2) w_1 + \frac s2 w_2 + \frac 12 \gamma_1 (w_1 + w_2),\\
\p_r  w_2 = \frac s2 w_1 - (1+\frac s 2) w_2 - \frac 12 \gamma_1 (w_1 + w_2),\\
\p_r s = - s^2. 
\end{cases} \ee 
Clearly $(0,0, 0)$ is an unstable equilibrium of the ODE system with $w_1$, $w_2$, and $s$ being in the unstable, stable, and the center directions, respectively. Therefore there exists a $C^{k_0}$ center-stable manifold $W^{\mathrm{cs}}$ in a neighborhood of $(0,0,0)$ given by  a graph 
\[
w_1 = \phi(w_2, s), \; \text{ with } \phi \in C^{k_0} \text{ and } \phi (0,0) =0, \quad \nabla \phi(0, 0) =0.
\]
Even though the center-stable manifold $W^{\mathrm{cs}}$ is usually not unique, the subset $W^{\mathrm{cs}} \cap \{s\ge 0\}$ is indeed unique because of its positive invariance under the ODE flow. 
Due to Assumption \ref{assumption:A}, both the orbit corresponding to $U$ as well as the trivial state $(0, 0, s = \frac 1r)$ converge to $(0, 0, 0)$ as $r\to+ \infty$.  Hence they both belong to $W^{\mathrm{cs}}$. This implies 
\[
\phi(0, s)=0, \; \text{ and thus } w_1 = \phi(w_2, s) = O\left(|w_2| (|s| +|w_2|)\right), \quad |w_2|,\, |s|\ll1.
\]
Therefore the only orbit on $W^{\mathrm{cs}}$ intersecting $\{w_2=0\}$ is the one corresponding to the trivial solution. On $W^{\mathrm{cs}}$, the $w_2$ equation in \eqref{E:blow-up} and the above properties of $\phi$ yield  
\[
\left|\p_r w_2  + \left(1+\frac s2 -\frac s2 \phi_{w_2} (0, s) \right) w_2\right|  = O(w_2^2), \quad { |w_2|,\, |s|\ll1.}
\]
Using this, we first calculate that  
\[
\frac d{dr} (e^r w_2^2) = \big(1 - 2(1+\frac s2 -\frac s2 \phi_{w_2} (0, s)) + O(|w_2|)\big) e^r w_2^2 \le 0, \quad { |w_2|,\, |s|\ll1.}
\]
Therefore $e^{\frac r2} |w_2|$ is decreasing in $r$ for $1\ll r$ and thus $\int_{r_0}^\infty w_2 \, \diff r$ converges absolutely. Moreover the above estimate of $\p_r w_2$ on $W^{\mathrm{cs}}$ further implies  
\[
w_2 (r) = \left(\frac {r_0}r \right)^{\frac 12} e^{r_0-r} e^{\tilde w(r)} w_2(r_0), \quad \tilde w(r) = \int_{r_0}^r \left[ \frac 1{2r'}  \phi_{w_2} \left(0, \frac 1{r'}\right) + O\left(|w_2(r')|\right) \right] \, \diff r'.  
\]
Since $\phi_{w_2} (0, s)= O(s)$, the above estimate implies that $\lim_{r\to +\infty} r^{\frac 12} e^r w_2(r)$ exists and belongs to $(0, \infty)$, which along with the fact that $w_1 = O\left(|w_2| (|\frac 1r| +|w_2|)\right)$ yields \eqref{eq:rate of decay} for $k=0,1$. 
\end{proof}

\begin{figure}
\centering
    \subfloat[$\tau=0$]{{\includegraphics[scale=0.80]{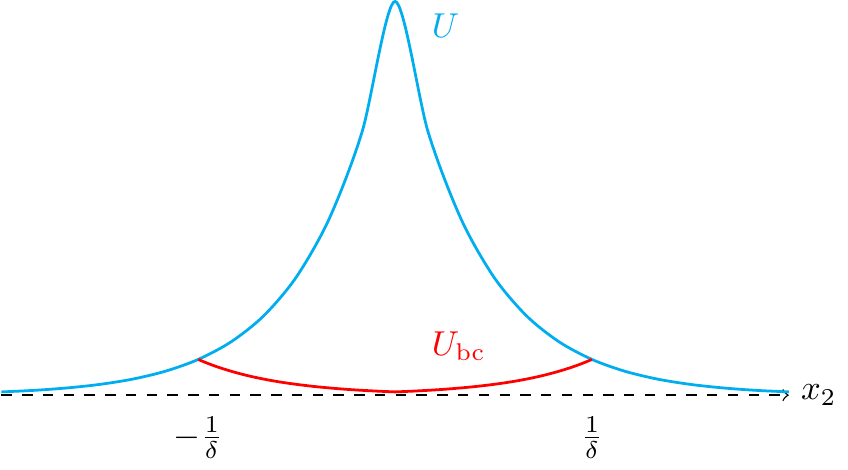} }}
    \qquad
    \subfloat[$\tau \in (0,\tfrac{1}{3})$]{{\includegraphics[scale=0.80]{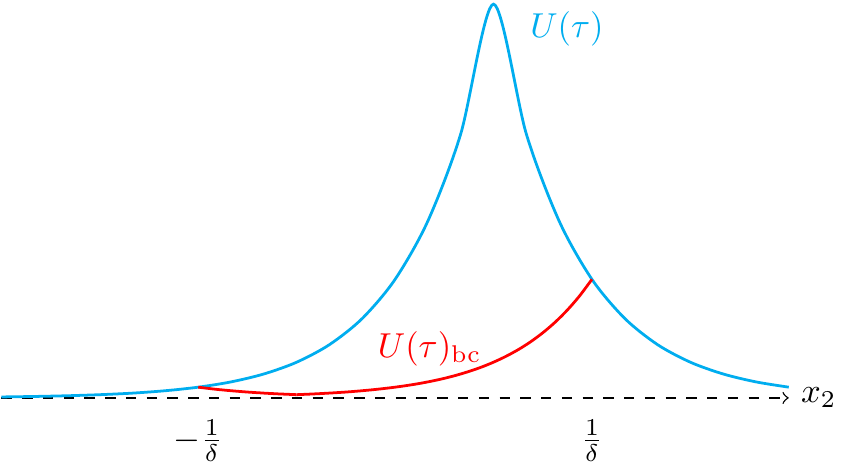} }}
    \caption{Graphs of $U$ and $U_\bc$ along the line $x_1 = 0$.  On the left, $U(x)$ is centered at the origin; on the right, it is shifted closer to the free surface. See also Corollary~\ref{C:Ubc}.} 
    \label{U Ubc figure}
\end{figure}

The following corollary will be used to analyze the boundary correction operator. For this and the coming results, especially Corollary~\ref{C:Ubc}, it can be good to consult Figure~\ref{U Ubc figure}. Note, in particular, that the estimate below essentially concerns the behavior of \(U\) \emph{outside} of the slab \(\slab_\delta\) (on the slab reflected over its own boundaries, modulo the translation \(\tau\).)

\begin{corollary} \label{C:Unorms}
For any $\tau \in [-\frac 13, \frac 13]$, $0\le k \le k_0+1$, and $0\le k'\le k_0+2$, $U(\tau)$ and $\p_{x_2} U(\tau)$ satisfy
\[
\left |U(\placeholder, \pm \tfrac 2\delta -\placeholder, \tau) \right|_{H^{k'} (\slab_\delta)}, \, \left|\p_{x_2} U(\placeholder, \pm \tfrac 2\delta -\placeholder, \tau) \right|_{H^{k} (\slab_\delta)} \eqsim \delta^{\frac 14} e^{-\frac {1 \mp \tau}\delta}.
\]
\end{corollary}

\begin{proof}
We shall focus on $\p_{x_2} U(\placeholder,  \frac 2\delta -\placeholder, \tau)$ as the others can be handled similarly. Consider the following subset $S$ of $\slab_\delta$
\begin{equation} \label{E:S}
S = \left\{ x\in \slab_\delta :  x_1^2 + (\tfrac {2-\tau}\delta - x_2)^2 < (\tfrac {3-\tau}\delta)^2 \right\}. 
\end{equation}
Let $(\rho, \beta)$ be the polar coordinates of $(x_1, \frac {2-\tau}\delta - x_2)$ so that 
\[
S = \{ (\rho, \beta) : \rho \in (\tfrac {1-\tau}\delta, \tfrac {3-\tau}\delta), \; \beta \in \left(\beta_0(\rho), \pi - \beta_0(\rho) \right) \}, \quad \beta_0(\rho) =  \sin^{-1}\left( \tfrac {1-\tau}{\delta \rho}\right). 
\] 
Since $|\tau| \le \frac 13$, we have $\sin \beta \eqsim 1$ in $S$, and 
\begin{equation} \label{E:beta}
\frac \pi2 - \beta_0 (\rho) = \sin^{-1} \left(1- \left(\tfrac {1-\tau}{\delta \rho}\right)^2\right)^{\frac 12} 
\eqsim (\delta \rho -1+\tau)^{\frac 12} \qquad \textrm{for all }  \rho \in (\tfrac {1-\tau}\delta, \tfrac {3-\tau}\delta).  
\end{equation}
Along with Proposition \ref{prop:sign of Uy}, this implies  
\[\begin{split}
I + II & :=  \left( \int_S + \int_{\slab_\delta\backslash S} \right) \left(x_1^2 + \left(\tfrac {2-\tau}\delta -x_2\right)^2\right)^{-\frac 12} e^{-2\left(x_1^2 + (\frac {2-\tau}\delta -x_2)^2\right)^{\frac 12}} \, \diff x \\
& \gtrsim  \left|\p_{x_2} U(\placeholder, \tfrac 2\delta -\placeholder, \tau) \right|_{H^{k_0+1} (\slab_\delta)}^2 
\gtrsim \left|\p_{x_2} U(\placeholder, \tfrac 2\delta -\placeholder, \tau)\right|_{L^2 (\slab_\delta)}^2 \gtrsim I. 
\end{split}\]
Again, it follows from Proposition \ref{prop:sign of Uy} and \eqref{E:beta} that
\[
I \eqsim \int_{\frac {1-\tau}\delta}^{\frac {3-\tau}\delta} \int_{\beta_0(\rho)}^{\pi-\beta_0(\rho)} e^{-2\rho} \,  \diff \beta \, \diff \rho \eqsim  \int_0^{\frac 2\delta} (\delta \rho')^{\frac 12} e^{-\frac {2(1-\tau)}\delta -2 \rho'} \, \diff \rho' \eqsim   \delta^{\frac 12} e^{-\frac {2(1-\tau)}\delta},
\]
while 
\[
II \lesssim \int_{|x| \ge \frac {3-\tau}\delta} |x|^{-1} e^{-2|x|} \, \diff x \lesssim \int_{\frac {3-\tau}\delta}^\infty e^{-2\rho} \, \diff \rho \eqsim e^{-\frac {2(3-\tau)}\delta}.
\]
This completes the proof of the corollary.
\end{proof}
In order to estimate the boundary correction operator defined in \eqref{definition bc}, we will need the following auxiliary lemma.  

\begin{lemma} \label{L:AuxBC}
Suppose $k \ge 2$ is an integer, $|\tau| \le \frac 13$, and $h\in C^k (\R^2, \R)$ satisfies 
\[
|\p^j h (x)|  \lesssim (1+|x - \tfrac \tau\delta e_2|)^{-\frac 12} e^{-|x- \frac \tau\delta e_2|} \quad \textrm{for } 0 \leq j \leq k,
\]
and 
\[
|\p^j (1-\Delta) h(x)| \lesssim (1+|x - \tfrac \tau\delta e_2|)^{-1} e^{-2|x- \frac \tau\delta e_2|} \quad  \textrm{for }  0 \leq j \leq k-2.
\]
Then 
\[
v(x_1, x_2)  := \left(h|_{\p \slab_\delta} \right)_\bc (x_1, x_2) -  h\left(x_1, \tfrac 2\delta -x_2\right) - h\left(x_1, -\tfrac 2\delta -x_2\right)
\]
satisfies 
\[
|v|_{H^k (\slab_\delta)}  \lesssim \delta^{\frac 34} e^{-\frac {2(1-|\tau|)}\delta}.
\]
\end{lemma}
Intuitively, this says that the boundary correction of $h$ is, to leading order, found by subtracting the reflections of $h$ over the top and bottom boundaries of the slab.  
\begin{proof}
From the definition of $\bc$ in \eqref{definition bc}, we see that $v$ satisfies 
\[ \begin{cases}
(1-\Delta) v (x_1, x_2) = -(1- \Delta) h\left(x_1, \frac 2\delta -x_2\right) - (1- \Delta)h \left(x_1, -\frac 2\delta -x_2\right) \qquad  \textrm{in } \slab_\delta, \\
v|_{x_2 = \pm \frac 1\delta} = - h(x_1, \mp \frac 3\delta).
\end{cases} \]
One can immediately deduce the energy estimate 
\[
|v|_{H^k (\slab_\delta)}  \lesssim \sum_\pm \left|(1-\Delta)h (\placeholder, \pm \tfrac 2\delta -\placeholder) \right|_{H^{k-2} (\slab_\delta)} + |h|_{H^{k-\frac 12} (\{ |x_2|= \frac 3\delta\})}. 
\]
An upper bound of the first term on the right-hand side above can be obtained much as in the proof of Corollary~\ref{C:Unorms}, and so we only provide a sketch and focus on the ``$+$'' case.  Let $S$ be given as in \eqref{E:S}, and split the slab $\slab_\delta = S \cup (\slab_\delta \setminus S)$.  From the properties assumed on $h$, we see that  
\begin{align*}
&\left|(1-\Delta)h (\placeholder, \tfrac 2\delta -\placeholder) \right|_{H^{k-2} (\slab_\delta)}^2 \lesssim  \left(\int_S + \int_{\slab_\delta\setminus S} \right) \left(x_1^2 + (\tfrac {2 - \tau}\delta -x_2)^2\right)^{-1} e^{-4\left(x_1^2 + (\frac {2 - \tau}\delta -x_2)^2\right)^{\frac 12}} \, \diff x   \\
 & \qquad \lesssim \int_{\frac {1-\tau}\delta}^{\frac {3-\tau}\delta} \int_{\beta_0(\rho)}^{\pi-\beta_0(\rho)} \rho^{-1} e^{-4\rho} \, \diff \beta \, \diff \rho + \int_{|x| \ge \frac {3-\tau}\delta} |x|^{-2} e^{-4|x|} \,  \diff x \\
 & \qquad \lesssim  \delta^{\frac 12} e^{-\frac {4(1- \tau)}\delta} \int_0^{\frac 2\delta} (\rho')^{\frac 12} (\frac {1-\tau}\delta + \rho')^{-1} e^{-4\rho'} {\diff} \rho' + \int_{\frac {3-\tau}\delta}^\infty \rho^{-1} e^{-4\rho} {\diff}\rho \lesssim  \delta^{\frac 32} e^{-\frac {4(1- \tau)}\delta}.   
\end{align*}
The $H^{k-{ \frac{1}{2}}}(\p \slab_{3/\delta})$ norm can be estimated by interpolating it between $H^k$ and $H^{k-1}$ and then appealing to the assumptions on $h$: 
\begin{equation} \label{E:exp2}  \begin{split}
 |h|_{H^{k-\frac 12} (\{ |x_2|= \frac 3\delta\})}^2 & \lesssim \left(\int_0^{\delta^{-\frac 12}} + \int_{\delta^{-\frac 12}}^\infty\right) \left(x_1^2 + \tfrac {(3- |\tau|)^2}{\delta^2}\right)^{-\frac 12} e^{-2\left(x_1^2 + \frac {(3- |\tau|)^2}{\delta^2}\right)^{\frac 12}} \,  \diff x_1 \\
 & \lesssim \delta^{\frac 12} e^{-\frac {2(3- |\tau|)}\delta} + \int_{\left(\frac {\delta + (3- |\tau|)^2}{\delta^2}\right)^{\frac 12}}^\infty x_1(s)^{-1} e^{-2s} \,  \diff s \lesssim  \delta^{\frac 12} e^{-\frac {2(3- |\tau|)}\delta},
\end{split}\end{equation}
where the substitution $x_1(s) = (s^2 - \tfrac {(3- |\tau|)^2}{\delta^2})^{\frac 12}$ was used to evaluate the integral on $[\delta^{-\frac 12}, \infty)$. Combining the above inequalities concludes the proof of the lemma. 
\end{proof}

Lemma~\ref{L:AuxBC} is mainly applied to $U_\bc$ and $\Uybc$ for $|\tau|\le \frac 13$. In fact, \eqref{eq:ground state} yields  
\[
(1-\Delta )\p_{x_2} U= \left(1-\gamma'(U)\right) \p_{x_2}U,
\]
and so the assumption that $\gamma'(0)=1$ together with Proposition~\ref{prop:sign of Uy} ensures that $U$ and $\p_{x_2}U$ satisfy the hypotheses of Lemma \ref{L:AuxBC}.  Therefore, in addition to Corollary \ref{C:Unorms} we obtain the following estimates, which will be essential to us later.

\begin{corollary} \label{C:Ubc}
For any $\tau \in [-\frac 13, \frac 13]$, $U(\tau)_\bc$ and $(\p_{x_2} U)(\tau)_\bc$ satisfy
\begin{align*}
& \left| U(\tau)_\bc - U(\placeholder, \tfrac 2\delta -\placeholder, \tau) - U(\placeholder, -\tfrac 2\delta -\placeholder, \tau) \right|_{H^{k_0+2} (\slab_\delta)}  \lesssim \delta^{\frac 34} e^{-\frac {2(1-|\tau|)}\delta}, \\
& \left| ( \p_{x_2} U)(\tau)_\bc - (\p_{x_2} U)(\placeholder, \tfrac 2\delta -\placeholder, \tau) - (\p_{x_2} U)(\placeholder, -\tfrac 2\delta -\placeholder, \tau)\right|_{H^{k_0+1} (\slab_\delta)}  \lesssim \delta^{\frac 34} e^{-\frac {2(1-|\tau|)}\delta},\\
& \left| U(\tau)_\bc \right|_{H^{k_0+2} (\slab_\delta)},~  \left| (\p_{x_2} U)(\tau)_{\bc} \right|_{H^{k_0+1} (\slab_\delta)} \eqsim \delta^{\frac 14} e^{-\frac {1 - |\tau|}\delta}.  
\end{align*}
\end{corollary}

\begin{proof}
The first two inequalities follow directly from Proposition~\ref{prop:sign of Uy} and Lemma \ref{L:AuxBC}. To obtained the estimate on $U(\tau)_{bc}$ based on the first inequality and Corollary \ref{C:Unorms}, we only need to show the almost orthogonality 
\begin{equation} \label{E:InnerP-1}
\left| \left(  U(\placeholder, \tfrac 2\delta -\placeholder, \tau), U(\placeholder, -\tfrac 2\delta -\placeholder, \tau) \right)_{H^{k_0+2} (\slab_\delta)}\right| \ll \delta^{\frac 12} e^{-\frac {2(1 - |\tau|)}\delta}. 
\end{equation}
In fact, for 
\[
b_1, b_2 \in \R, \; |b_1| -1\in [\frac 13, 3], \; |b_1-b_2| \ge \frac 13,\quad x\in \slab_\delta, \; |x_1| \le \frac 4\delta,
\]
due to the convexity of $t \mapsto \jbracket{t}$, there exists $\sigma>0$ independent of  $b_1, b_2$, and small $\delta>0$ such that 
\begin{equation} \label{E:Exp-1} \begin{split}
\left|\tfrac {b_1}\delta e_2 -x \right| + \left|x- \tfrac {b_2}\delta e_2 \right| & \ge \left|x_2- \tfrac {b_2}\delta \right| + \left|\tfrac {b_1}\delta - x_2\right| \left( 1+ \left|\tfrac {b_1}\delta - x_2\right|^{-2} x_1^2\right)^{-\frac 12} \\
& \ge   \left|x_2- \tfrac {b_2}\delta\right| + \left|\tfrac {b_1}\delta - x_2\right| + \sigma \left|\tfrac {b_1}\delta - x_2\right|^{-1} x_1^2 \ge \tfrac {|b_1+b_2|}\delta + \sigma \delta x_1^2. 
\end{split}\end{equation}
It is also clear that 
\be \label{E:Exp-2}
1+ \left|\tfrac {b_1}\delta e_2 - x \right| \eqsim \delta^{-1}, \quad 1+\left|x-\tfrac {b_2}\delta e_2\right| \gtrsim 1+ \left|x_2 - \tfrac {b_2}\delta\right|. 
\ee
Applying these inequalities to $b_1 = 2-\tau$ and $b_2=-(2+\tau)$ we obtain 
\[\begin{split}
 & \int_{\slab_\delta} (1+|x-\tfrac {2-\tau}\delta e_2|)^{-\frac 12} (1+ |x+\tfrac { 2+ \tau}\delta e_2|)^{-\frac 12} e^{-|x+\frac {2+ \tau}\delta e_2| -|x-  \frac { 2 - \tau}\delta e_2|}  \, \diff x
\\
 & \qquad \lesssim  \delta \left( \int_{|x_1| \le \frac 4\delta} + \int_{|x_1|\ge \frac 4\delta} \right) \int_{-\frac 1\delta}^{\frac 1\delta} e^{-|x-\frac \tau\delta e_2| - | \frac { 2 - \tau}\delta e_2 - x|} \, \diff x_2\, \diff x_1\\
& \qquad \lesssim  \int_{|x_1| \le \frac 4\delta}  e^{-\frac {4-2|\tau|}\delta - \sigma \delta x_1^2} \, \diff s \, \diff x_1 + \int_{\frac 4\delta}^\infty e^{-2x_1} \, \diff x_1 \lesssim e^{-\frac 3\delta}.
\end{split}\]
Together with Proposition~\ref{prop:sign of Uy} it immediately implies \eqref{E:InnerP-1} and completes the proof of the corollary. 
\end{proof}

\subsection*{Estimating the nonlinearity} 
Finally, we give some estimates of the nonlinearities $F$ and $G$ occurring in the reformulated water wave problem \eqref{perturbed problem}.

\begin{lemma}
\label{equivalence lemma}   For \(\gamma\) as in Assumptions~\ref{assumption:A} and~\ref{assumption:B} and any integer $2\le k \le k_0$, there exists $\sigma \in (0, 1)$ depending only on $g$ and $\alpha$, such that the operators $F$ and $G$ given in \eqref{def F} and \eqref{def G} satisfy 
\begin{align*}
& F: (-\tfrac 13, \tfrac 13) \times H_{\mathrm{e}}^k (\slab_\delta) \times H_{\mathrm{e}}^{k} (\R) \to H_{\mathrm{e}}^{k-1} (\slab_\delta) \text{ is } C^{k_0-k+1} \text{ in } u, \Gammas, \; C^{k_0-k} \text{ in } \tau;\\
& G : \left(-\tfrac 13, \tfrac 13 \right) \times  H_{\mathrm{e}}^k (\slab_\delta) \times B_{\sigma} \left(H_{\mathrm{e}}^{k} (\R) \right) \to H_{\mathrm{e}}^{k'} (\R) \text{ is } C^\infty \text{ in } u, \Gammas, \; C^{k_0-k'+1} \text{ in } \tau,
\end{align*}
where $B_{\sigma} \left(H_{\mathrm{e}}^{k} (\R) \right)$ is the ball in $H_{\mathrm{e}}^{k} (\R)$ centered at $0$ with radius $\sigma$ and $k'=k-\frac 32$ if $k>2$ and $k'$ can be any number smaller than $k-\frac 32$ if $k=2$. 
Moreover, for any $\sigma_u \in (0, 1)$, $\sigma_\Gamma \in (0, \sigma)$, $\tau \in (-\frac 13, \frac 13)$, $u \in B_{\sigma_u} \left(H_{\mathrm{e}}^k (\slab_\delta)\right)$ and $\Gammas \in B_{\sigma_\Gamma} \left(H_{\mathrm{e}}^{k} (\R)\right)$, we have 
\[ \begin{aligned}
|D_u F |_{\mathcal{L}\left(  H_{\mathrm{e}}^k (\slab_\delta), H_{\mathrm{e}}^{k-2} (\slab_\delta) \right)} & \lesssim \sigma_u + \delta^{-1}\sigma_\Gamma + \delta^{\frac 14} e^{-\frac {1-|\tau|}\delta}, \\
 |D_{\Gammas} F |_{\mathcal{L}\left(  H_{\mathrm{e}}^k (\R), H_{\mathrm{e}}^{k-2} (\slab_\delta) \right)} & \lesssim  \delta^{-1},     \\
|F(\tau, 0, 0)|_{H_{\mathrm{e}}^{k-2} (\slab_\delta)} & \lesssim  \delta^{\frac 14}\left| \log \delta \right|^{\frac 12} e^{-\frac {2(1-|\tau|)}\delta},
\end{aligned}\]
and 
\[ \begin{aligned}
 |D_u G|_{\mathcal{L}\left(  H_{\mathrm{e}}^k (\slab_\delta), H_{\mathrm{e}}^{k-2} (\R) \right)} & \lesssim \delta^{\frac{1}{2} -k} \sigma_u + \delta^{\frac{3}{4} -k} e^{-\frac {1-\tau}\delta}, \\
 |D_{\Gammas} G |_{\mathcal{L}\left(  H_{\mathrm{e}}^k (\R), H_{\mathrm{e}}^{k-2} (\R) \right)} & \lesssim \delta^{\frac{1}{2}-k} (\sigma_u^2 + \delta^{\frac 12} e^{-\frac {2(1-\tau)}\delta}), \\
 |G(\tau, 0, 0)|_{H_{\mathrm{e}}^{k-2} (\R)} & \lesssim \delta^{1-k} e^{-\frac {2(1-\tau)}\delta}.
\end{aligned}\]
\end{lemma}

\begin{proof}
Verifying the smoothness of $F$ and $G$ is tedious but straightforward.  The argument is based on (i) standard regularity results on products in Sobolev spaces, properties of the harmonic extension, the trace theorem, and (ii) the $C^{k_0 -l'}$ smoothness of the mapping $u  \in H^l\mapsto \gamma \circ u \in H^{l'}$ for a given $\gamma \in C^{k_0}$, which holds for $l'\le l$ and $l> \frac n2 +1$ in $n$ dimensions. The limitation on the smoothness of $F$ and $G$ with respect to $\tau$ is only due to the $C^{k_0+2}$ dependence of $U_\bc$ in $\tau$. The small $\sigma>0$ is chosen such that the denominator in the definition of $G$ is bounded away from zero and $A(\Gammas)$ has a bounded inverse, which can be done independent of $|\tau|\le \frac 13$ and small $\delta>0$. We omit the details and focus on the quantitative estimates related to $F$ and $G$. In what follows, let $\Gamma= \Gamma_1 + \I \Gamma_2\in H_{\textrm{e}}^{k+1/2} (\slab_\delta)$ be the conformal mapping determined by $\Gammas$ through \eqref{E:conformal} and \eqref{def Gammas}.  Note that this involves just the harmonic extension \eqref{E:HarmExt} and harmonic conjugate operators. 

From the definition of $F$, 
\begin{align*}
F(\tau, 0, 0) & =\gamma(U - U_\bc) - \gamma(U) + U_\bc \\
&  = \int_0^1 \!\! \int_0^1 \gamma''( s_2 U - s_1s_2 U_\bc) (s_1 U_\bc  -U ) U_{\bc} \, \diff s_2 \, \diff s_1,
\end{align*}
which, along with Corollary~\ref{C:Ubc}, implies that
\be
\begin{split}
|F(\tau, 0, 0)|_{H_{\textrm{e}}^{k-2} (\slab_\delta)} \lesssim & |U U_{\bc}|_{H_{\textrm{e}}^{k-2} (\slab_\delta)} + |U_{\bc}^2|_{H_{\textrm{e}}^{k-2} (\slab_\delta)} \\
\lesssim & \left| \sum_{\pm} U(\placeholder, \pm \tfrac{2}{\delta} - \placeholder) U \right|_{H_{\textrm{e}}^{k-2} (\slab_\delta)} 
+ \delta^{\frac 12} e^{-\frac {2(1-|\tau|)}\delta}. 
\end{split}
\label{prelim F(tau,0,0) bound}
\ee
Without loss of generality, we only need to consider the ``+'' term in the summation. 
According to Assumption \ref{assumption:A} and Proposition \ref{prop:sign of Uy}, for any $0\le j\le k-2$ and $x \in \slab_\delta$, 
\[
\left| \p^j \left( U(x) U (x_1, \tfrac 2\delta -x_2) \right)\right| \lesssim (1+|x-\tfrac \tau\delta e_2|)^{-\frac 12} |\tfrac { 2- \tau}\delta e_2 - x|^{-\frac 12} e^{-(|x-\frac \tau\delta e_2| + | \frac { 2 - \tau}\delta e_2 - x|)}
\]
which can be estimated much as \eqref{E:InnerP-1}. Applying \eqref{E:Exp-1} and \eqref{E:Exp-2} to $b_1 = 2-\tau$ and $b_2=\tau$, we have 
\[\begin{split}
 &\int_{\slab_\delta} (1+|x-\tfrac \tau\delta e_2|)^{-1} |\tfrac { 2- \tau}\delta e_2 - x|^{-1} e^{-2(|x-\frac \tau\delta e_2| + | \frac { 2 - \tau}\delta e_2 - x|)}  \, \diff x
\\
& \qquad \lesssim  \left( \int_{|x_1| \le \frac 4\delta} + \int_{|x_1|\ge \frac 4\delta} \right)  \int_{-\frac 1\delta}^{\frac 1\delta} (1+|x-\tfrac \tau\delta e_2|)^{-1} |\tfrac { 2- \tau}\delta e_2 - x|^{-1} e^{-2(|x-\frac \tau\delta e_2| + | \frac { 2 - \tau}\delta e_2 - x|)} \, \diff x_2 \, \diff x_1\\
& \qquad \lesssim  \int_{|x_1| \le \frac 4\delta}  \int_0^{\frac {1+|\tau|}\delta} \delta (1+ s)^{-1} e^{-2(\frac {2(1-|\tau|)}\delta + \sigma \delta x_1^2)} \, \diff s \, \diff x_1 + \int_{|x|\ge \frac 4\delta} \delta^2 e^{-2|x|} \, \diff x,
\end{split}\]
and so we have that 
\be \label{E:F0}
\int_{\slab_\delta} \left(1+|x-\tfrac \tau\delta e_2| \right)^{-1} |\tfrac { 2- \tau}\delta e_2 - x|^{-1}| e^{-2(|x-\frac \tau\delta e_2| + | \frac { 2 - \tau}\delta e_2 - x|)} \, \diff x \lesssim  \delta^{\frac 12} \left| \log \delta \right| e^{-\frac {4(1-|\tau|)}\delta}.
\ee
This further implies 
\[
|U (\placeholder, \tfrac 2\delta -\placeholder) U |_{H^{k-2} (\slab_\delta)}^2 \lesssim  \delta^{\frac 12} | \log \delta| e^{-\frac {4(1-|\tau|)}\delta}
\]
which, with \eqref{prelim F(tau,0,0) bound}, furnishes the desired estimate of $F(\tau, 0, 0)$.

Next, observe that, for any $\tilde u \in H_{\mathrm{e}}^k(\slab_\delta)$ with $k \geq 2$,
\begin{equation}\label{eq:Du F estimate}
\begin{split}
& D_u F (\tau, u, \Gammas) \tilde u  = \left(|1+\Gamma^\prime(\delta \placeholder )|^2 \gamma' (u+U-U_{\textrm{bc}})- \gamma^\prime(U)\right) \tilde u \\
& \qquad = \left(\left(2 \Gamma_1' (\delta \placeholder)+ |\Gamma'(\delta \placeholder)|^2\right)  \gamma' (u+U-U_{\textrm{bc}}) + (u - U_{\bc}) \int_0^1 \gamma'' \left(U + s(u - U_{\bc}) \right) \, \diff s \right) \tilde u. 
\end{split}
\end{equation}
Now, for any $s$ we have the the scaling identity
\[ | f(\delta \placeholder) |_{\dot H^s(\slab_\delta)} = \delta^{s-1} | f |_{\dot H^s(\slab_1)},\]
and, for $k - \frac{1}{2} \geq \frac{3}{2}$ and $|\Gammas|_{H^k} < 1$, it holds that
\[ \left | 2 \Gamma_1^\prime(\delta \placeholder) + |\Gamma^\prime(\delta \placeholder) |^2 \right|_{H^{k-\frac{1}{2}}(\slab_\delta)} \lesssim \delta^{-1} |\Gammas|_{H^k(\mathbb{R})}.\]
Thus, the $H^{k-2} (\slab_\delta)$ norm of the last line of \eqref{eq:Du F estimate} has the upper bound 
\[
O\left( \delta^{-1} |\Gammas|_{H^{k}  (\R)} + |u|_{H^{k}  (\slab_\delta)} + |U_{\bc}|_{H^{k}  (\slab_\delta)}\right) |\tilde u|_{H^{k-2} (\slab_\delta)}.
\] 
Corollary \ref{C:Ubc} therefore gives the claimed estimate of $D_u F$. 

Regarding $D_{\Gammas} F$, we have for any $\Gammast \in H_{\mathrm{e}}^k(\mathbb{R})$ that 
\[
D_{\Gammas} F (\tau, u, \Gammas) \Gammast = 2 \left( \tilde \Gamma_1' (\delta \placeholder) + \Gamma' (\delta \placeholder) \cdot \tilde \Gamma' (\delta \placeholder) \right) \gamma (u+U-U_{\bc}), 
\]
where $\tilde \Gamma = \tilde \Gamma_1 + \I \tilde \Gamma_2$ is the complex holomorphic function determined by $\Gammast$.  The estimate on $D_{\Gammas} F$ follows from this expression and the scaling of the Sobolev norms explained above.  

From the definition of $G$, one can directly compute 
\be \label{E:G0}
G(\tau, 0, 0)
= \frac 1{2 \delta^2} (\p_{x_2} (U - U_{\bc})) \left(\tfrac \placeholder \delta, \tfrac 1\delta \right)^2. 
\ee
Recall the one-dimensional scaling property,
\[
|f\left(\delta^{-1} \placeholder \right)|_{\dot H^s (\R)} = \delta^{-s+\frac 12} |f|_{\dot H^s (\R)},
\]
which holds for all $s \in \mathbb{R}$.  
The term $(\p_{x_2} (U - U_{\bc})) (\placeholder, \tfrac 1\delta)$ can be estimated by approximating $\p_{x_2} U_{\bc}(\placeholder, \frac 1\delta)$ by $-\p_{x_2} U(\placeholder, \frac 1\delta)$. From Corollary \ref{C:Ubc} and the trace theorem, we then have, for any $k_0+\frac 12\ge k\ge 0$, 
\be \label{E:temp-a}
\left |\p_{x_2} (U - U_{\bc}) (\placeholder, \tfrac 1\delta)
-2(\p_{x_2} U) (\placeholder, \tfrac 1\delta) \right|_{H^k (\R)} \lesssim |(\p_{x_2} U) ( \placeholder, -\tfrac 3\delta)|_{H^k (\R)} + \delta^{\frac 34} e^{-\frac {2(1-|\tau|)}\delta}.
\ee
Using Proposition \ref{prop:sign of Uy} and the change of variables $x_1 (\rho) = (\rho^2 - (\frac {1-\tau}\delta)^2)^{\frac 12}$ , we can estimate $(\p_{x_2} U) (\placeholder, \frac 1\delta)$ while 
the terms on the right hand-side are obviously much smaller,
\be \label{E:temp-0} \begin{split}
\left |(\p_{x_2} U) ( \placeholder, \tfrac 1\delta) \right|_{H^k (\R)}^2 \lesssim & \left(\int_0^{\delta^{-\frac 12}} + \int_{\delta^{-\frac 12}}^\infty\right) \left(x_1^2 + (\tfrac {1-\tau}\delta)^2 \right)^{-\frac 12} e^{-2(x_1^2 + (\frac {1-\tau} \delta)^2)^{\frac 12}} \, \diff x_1\\
\lesssim & \delta^{\frac 12} e^{-\frac {2(1-\tau)}\delta} + \int_{\frac {((1-\tau)^2 +\delta)^{\frac 12}}\delta }^\infty \frac {\diff \rho} {x_1(\rho) e^{2\rho}}\lesssim \delta^{\frac 12} e^{-\frac {2(1-\tau)}\delta}.  
\end{split} \ee
This implies that
\be \label{E:temp-1}
\left|(\p_{x_2} (U - U_{\bc})) (\placeholder, \tfrac 1\delta) \right|_{H^k (\R)} \lesssim \delta^{\frac 14} e^{-\frac {1-\tau}\delta}.
\ee
We therefore obtain the estimate on $G(\tau, 0, 0)$ from the scaling property as 
\[
|G(\tau,0, 0)|_{H^{k-2} (\R)} \lesssim \delta^{\frac 12 -(k-2)-2} \left| \left((\p_{x_2} (U - U_{\bc})) (\placeholder, \tfrac 1\delta)\right)^2 \right|_{H^{k-2} (\R)} \lesssim \delta^{1-k} e^{-\frac {2(1-\tau)}\delta}.
\]  

Consider next the bound on $D_u G$.  It is easy to see from the definition of $G$ that,
\[
D_u G(\tau, u, \Gammas) \tilde u= \frac{1}{\delta^{2}} A(\Gammas)^{-1} \left[ \frac{(\p_{x_2} (u + U - U_\mathrm{bc})\p_{x_2} \tilde u)(\frac{\placeholder} \delta,\frac{1}{\delta}) }{(1 +|\Diff| \coth{(2|\Diff|)} \Gammas)^2 + {\Gammas^\prime}^2}    \right]. 
\]
By the trace theorem and \eqref{E:temp-1},
we have 
\begin{align*}
& \left| \left(\p_{x_2} (u + U - U_\mathrm{bc})\p_{x_2} \tilde u \right)({\placeholder},\tfrac{1}{\delta})\right|_{H^{k-2} (\R)} \\
& \qquad 
\lesssim  |\p_{x_2} (u + U - U_\mathrm{bc})({\placeholder},\tfrac{1}{\delta})|_{H^{k-\frac32} (\R)} |\p_{x_2} \tilde u ({\placeholder},\tfrac{1}{\delta})|_{H^{k-\frac32} (\R)} \lesssim \left(\sigma_u + \delta^{\frac 14} e^{-\frac {1-\tau}\delta} \right) |\tilde u|_{H^k (\slab_\delta)}.
\end{align*}
The desired bound on $D_u G$ then follows from the scaling property.

 Finally, for any $\Gammast \in H_{\mathrm{e}}^k(\mathbb{R})$, 
\begin{align*}
2\delta^2 D_{\Gammas} G(\tau, u, \Gammas) \Gammast = & \left( D_{\Gammas} (A(\Gammas)^{-1}) \Gammast \right) \Big[ \frac{(\p_{x_2} (u + U - U_\mathrm{bc}))(\frac{\placeholder} \delta,\frac{1}{\delta})^2 }{(1 +|\Diff| \coth{(2|\Diff|)} \Gammas)^2 + {\Gammas^\prime}^2}\Big] \\
&- 2A(\Gammas)^{-1}\Big[ \frac{(\p_{x_2} (u + U - U_\mathrm{bc}))(\frac{\placeholder} \delta,\frac{1}{\delta})^2 }{\left( (1 +|\Diff| \coth{(2|\Diff|)} \Gammas)^2 + {\Gammas^\prime}^2\right)^2}  \\
& \times \left(  (1 +|\Diff| \coth{(2|\Diff|)} \Gammas)|\Diff| \coth{(2|\Diff|)}  \Gammast + \Gammas' \Gammast '\right)\Big].
\end{align*}
Since $|u|_{H^k (\slab_\delta)} < \sigma_u$ and $|\Gammas|_{H^k (\R)} < \sigma_\Gamma< \sigma<1$ with $k \ge 2$, straightforwardly we obtain 
\begin{align*}
|D_{\Gammas} G(\tau, u, \Gammas) \Gammast|_{H^{k-2} (\R)} \lesssim & \delta^{-2} |(\p_{x_2} (u + U - U_\mathrm{bc}))(\tfrac{\placeholder} \delta,\tfrac{1}{\delta})^2|_{H^{k-2}(\R)}  |\tilde \Gammas|_{H^k (\R)} \\
\lesssim & \delta^{\frac 12-k} \left(|u|_{H^{k}(\slab_\delta)}^2 + |(\p_{x_2} (U - U_\bc))(\placeholder, \tfrac 1\delta)|_{H^{k-1}(\R)}^2 \right)  |\Gammast|_{H^k (\R)},
\end{align*}
where the scaling property and the trace theorem have been used. The estimate on $D_{\Gammas} G$ now follows immediately from \eqref{E:temp-1}. 
\end{proof}

One notices that $|D_{\Gammas} F(\tau, u, \Gammas)|$ is not small no matter how small $u$ and $\Gammas$ are. Fortunately this is an ``off diagonal term" in the linearization, which will be handled by a simple rescaling argument in Section \ref{proof main theorem sec}.

\section{Spectral properties}\label{sec:Spec}
Having the necessary estimates on $U$ and the boundary correction operator $\bc$ now in hand, we 
next consider the linear operator  
\[
 L_\tau =  -\Delta + \gamma^\prime(U(\tau)) \colon X_\delta^2 \to X_\delta^0,
 \] 
in the elliptic equation \eqref{perturbed u eq}. 
Recall that the Dirichlet boundary conditions on $\partial \slab_\delta$ are encoded in the definition of the space $X_\delta^2$, and that we usually suppress the dependence on the translation parameter \(\tau\) in the notation for \(U = U(\placeholder, \tau) = U(\tau)(\placeholder)\).  

The inherent difficulty here is that equation \eqref{eq:ground state} implies
\[
\Delta \partial_{x_2} U = \gamma^\prime(U) \partial_{x_2} U,
\] 
so that $\partial_{x_2} U$ is in the kernel of $-\Delta+\gamma^\prime(U)$ viewed as an operator with domain  $H_{\textrm{e}}^2(\mathbb{R}^2)$.  Working in the strip $\slab_\delta$ breaks the vertical translation symmetry and  eliminates this kernel direction.  It is therefore expected that $L_\tau$ will be invertible from $X_\delta^2 \to X_\delta^0$, and, indeed, this is proved in Lemma~\ref{invertibility L on U_0 perp}.  However, as $\delta \searrow 0$, heuristically the strip approximates $\mathbb{R}^2$, and so we cannot hope to obtain bounds for $L_\tau^{-1} \colon X_\delta^0 \to X_\delta^2$ that are \emph{uniform} in $\delta$.  Another way to see this is to note that 
\[ L_\tau \partial_{x_2} U = 0 \quad \textrm{in } \slab_\delta, \qquad \partial_{x_2} U  \eqsim \delta^{\frac{1}{2}} \e^{-\frac{1- |\tau|}{\delta}} \quad \textrm{on } \partial\slab_\delta,\] 
as \(U\) and its derivatives decay exponentially. Thus, \(L_\tau\) is nearly degenerate in the direction close to \(\partial_{x_2} U\). 

In order to proceed, it is therefore necessary to have detailed information about the behavior of $L_\tau$ as $\delta \searrow 0$.  We will prove that there is a positive (simple) eigenvalue $l=l(\delta, \tau)$ that is exponentially small in $\delta$ for fixed $\tau$, and whose eigenfunction $U_0$ approaches $\partial_{x_2} U$ as $\delta \searrow 0$.  In the orthogonal complement of $U_0$, the inverse of $L_\tau$ is bounded uniformly in $\delta$.  In the next section, we will make use of this fact to perform a Lyapunov--Schmidt type reduction to \eqref{perturbed u eq}, first solving the problem on a codimension $1$ subspace where $L_\tau$ is well-behaved, and then studying the reduced equation on the near-degenerate direction.

\subsection{An approximate eigenfunction}  \label{approximate eigen section}

As a preparation for proving the existence of $l$ and $U_0$, we first study the function 
\be
U_2(\placeholder,\tau) = \partial_{x_2} U(\placeholder,\tau) - (\partial_{x_2} U)(\placeholder,\tau)_\textrm{bc} \in X_\delta^2, \label{def U2} 
\ee
which results from taking $\partial_{x_2} U$ and perturbing it slightly so that the homogeneous boundary condition on $\partial \slab_\delta$ is satisfied, see Figure~\ref{U Ubc figure}.  In what follows, the dependence of $U_2$ on $\tau$ will be suppressed when there is no risk of confusion.   While $U_2$ is not likely to be an eigenfunction itself, we will see that it does help in identifying the asymptotically degenerate direction.  Observe that it solves the elliptic problem
\be \label{E:LtauU2}
\left\{ \begin{array}{rll} \left(-\Delta + \gamma^\prime(U) \right)  U_2  & = (1- \gamma^\prime(U))\Uybc & \textrm{in } \slab_\delta \\
U_2 & = 0 & \textrm{on } \partial \slab_\delta,\end{array} \right.
\ee
as \( (-\Delta + \gamma'(U)) \partial_{x_2} U =0\) 
and $\Delta \Uybc = \Uybc$ by the definition of the boundary correction operator. Recall also that $\gamma \in C^{k_0}$ according to Assumption \ref{assumption:A}.

\begin{lemma} \label{L:U2}
For $|\tau| \le \frac 13$, we have $|U_2|_{H^{k_0+1} (\slab_\delta)} \eqsim 1$ and 
\begin{align*}
|L_\tau U_2|_{H^{k_0-1} (\slab_\delta)}\lesssim \delta^{\frac 14} \left|\log{\delta}\right|^{\frac 12} e^{-\frac {2(1-|\tau|)}\delta}, \qquad  
0< (U_2, L_\tau U_2)_{L^2(\slab_\delta)} \eqsim \delta^{\frac 12} e^{-\frac{2(1-|\tau|)}{\delta}}.
\end{align*}
\end{lemma}

\begin{proof}
Simply from the exponential decay of $U(x)$ as $|x| \to \infty$, Proposition \ref{prop:sign of Uy}, Corollary~\ref{C:Ubc}, and its definition, it is clear that $U_2 = O(1)$ in $H^{k_0+1}$ for $|\tau| \leq \tfrac{1}{3}$. On the other hand, from \eqref{E:LtauU2} and Corollary~\ref{C:Ubc}, 
we obtain the estimate
\begin{equation} \label{E:LtauU2-1}
\left| L_\tau U_2 - \left(1-\gamma'(U) \right)\sum_\pm \p_{x_2} U (\placeholder,\pm  \tfrac 2\delta -\placeholder) 
\right|_{H^{k_0-1} (\slab_\delta)} \lesssim \delta^{\frac 34} e^{-\tfrac {2(1-|\tau|)}\delta}. 
\end{equation} 
Without loss of generality, we only need to consider the ``+'' case. 
According to Assumption~\ref{assumption:A} and 
Proposition~\ref{prop:sign of Uy}, for any $0\le k\le k_0 -1$ and $(x_1, x_2) \in \slab_\delta$, 
\[
\left| \p^k \left( \left(1-\gamma'(U(x)) \right)\p_{x_2} U (x_1, \tfrac 2\delta -x_2)\right)\right| \lesssim \frac{e^{-\left(|x-\tfrac \tau\delta e_2| + | \tfrac { 2 - \tau}\delta e_2 - x|\right)}
} { \left(1+|x-\tfrac \tau\delta e_2| \right)^{\frac 12} \left|\tfrac { 2- \tau}\delta e_2 - x \right|^{\frac 12}}
\]
and thus the claimed bound on $L_\tau U_2$ follows from \eqref{E:F0}.

Likewise, using the above estimate on $L_\tau U_2$ in conjunction with Corollary~\ref{C:Ubc}, \eqref{def U2}, and \eqref{E:LtauU2}, one can estimate 
\begin{align*}
(U_2, \, L_\tau U_2)_{L^2(\slab_\delta)} &= \left(\partial_{x_2} U - (\partial_{x_2} U)_\mathrm{bc} , \, L_\tau 
U_2 \right)_{L^2(\slab_\delta)}\\ 
&=  \left( (1- \gamma^\prime(U)) \partial_{x_2} U, \,  (\partial_{x_2} U)_\mathrm{bc}\right)_{L^2(\slab_\delta)} + O\big(e^{-\frac{3(1-|\tau|)}{\delta}}\big).
\end{align*}
We concentrate on the first term on the right-hand side, as it will clearly dominate as $\delta \searrow 0$.  Using the identities 
\[ (1- \gamma^\prime(U)) \partial_{x_2} U = (1- \Delta) \partial_{x_2} U, \qquad (1-\Delta) (\partial_{x_2} U)_\bc = 0,\]
and integrating by parts twice yields:
\begin{align*}
\int_{\slab_\delta} [(1- \gamma^\prime(U)) \partial_{x_2} U]   (\partial_{x_2} U)_\bc \, \diff x &= \int_{\slab_\delta} [(1- \Delta) \partial_{x_2} U]   (\partial_{x_2} U)_\mathrm{bc} \, \diff x \\
&=  \int_{\partial \slab_\delta}  \left(  - [\partial_{x_2}^2 U]   (\partial_{x_2} U)_\mathrm{bc} + \partial_{x_2} U \partial_{x_2} (\partial_{x_2} U)_\mathrm{bc} \right) N_2 \, \diff x_1 \\
& =   \int_{\partial \slab_\delta}  \left( (\partial_{x_2} U)_\mathrm{bc} \partial_{x_2} (\partial_{x_2} U)_\mathrm{bc} - [\partial_{x_2}^2 U]   \partial_{x_2} U \right) N_2 \, \diff x_1.
\end{align*}
The first of the boundary integrals we treat by integrating back to the interior domain \(\slab_\delta\) and using the definition of the boundary correction operator:
\begin{align*}
\int_{\partial \slab_\delta}  (\partial_{x_2} U)_\mathrm{bc} \partial_{x_2} (\partial_{x_2} U)_\mathrm{bc} N_2 \, \diff x_1 
 &= \int_{\slab_\delta} \left( |\nabla (\partial_{x_2} U)_\mathrm{bc} |^2 + (\partial_{x_2} U)_\mathrm{bc}^2 \right) \, \diff x  \eqsim \delta^{\frac 12} e^{-\frac {2(1 - |\tau|)}\delta},
 \end{align*}
where Corollary~\ref{C:Unorms} and~\ref{C:Ubc} are used in the above last step. The second integral we instead estimate by integrating into the \emph{outer} domain  $\slab_\delta^c$, where \(U\) and its derivatives are well defined and exponentially decaying in all radial directions. In analogy to above, we use the elliptic equation that  \( \partial_{x_2} U\) satisfies to find
\begin{align*}
-\int_{\partial \slab_\delta} (\partial_{x_2}^2 U) (   \partial_{x_2} U ) N_2 \, \diff x_1 &=  \int_{ \slab_\delta^c} \left( |\nabla (\partial_{x_2} U) |^2 + (\partial_{x_2} U) \Delta (\partial_{x_2} U) \right) \, \diff x\\
& = \int_{\slab_\delta^c} \left( |\nabla (\partial_{x_2} U) |^2 + \gamma^\prime(U) (\partial_{x_2} U)^2 \right) \, \diff x \\
& =  \int_{|x_2| \in (\frac 1\delta, \frac 3\delta)} \left( |\nabla (\partial_{x_2} U) |^2 +  (\partial_{x_2} U)^2 \right) \, \diff x + O\big(e^{-\frac{3(1-|\tau|)}{\delta}}\big)\\
&\eqsim 
 \delta^{\frac 12} e^{-\frac{2(1-|\tau|)}{\delta}},
\end{align*} 
where the last bound is from Corollary~\ref{C:Unorms}. Observe also that both boundary integrals are positive.
This implies the positivity of $(U_2, \, L_\tau U_2)_{L^2(\slab_\delta)}$ for $\delta$ small enough, and so the proof is complete. 
\end{proof} 

We wish to show that $L_\tau\colon X_\delta^2 \to X_\delta^0$ is well behaved as $\delta \searrow 0$ except in a one-dimensional near-degenerate direction that anticipates the kernel of $-\Delta + \gamma^\prime(U)$ on $\mathbb{R}^2$.  To be more precise, define the function spaces 
\[
X = H_{\textrm{e}}^2(\mathbb{R}^2), \qquad Y = L_{\textrm{e}}^2(\mathbb{R}^2).
\]
Then, under Assumption~\ref{assumption:B} we see that $-\Delta+ \gamma^\prime(U)\colon X \to Y$ has a one-dimensional kernel spanned by $\partial_{x_2} U$.  Note that here the even symmetry restriction {eliminates} the kernel direction generated by $\p_{x_1} U$.  Let $P = P(\tau)$ denote the orthogonal projection of $Y$ onto $\spn{\{\p_{x_2} U\}}$; abusing notation somewhat, we use the same symbol for the induced projection $X \to \spn{\{\p_{x_2} U\}}$. 

The following non-degeneracy result is a direct consequence of Assumption~\ref{assumption:B}. 
 
 \begin{lemma}[Non-degeneracy in $\mathbb{R}^2$] \label{nondegeneracy in R^2 lemma}
The operator  $-\Delta + \gamma^\prime(U) \colon (I-P)X \to (I-P) Y$
 is an isomorphism with bounds uniform in $|\tau| \le \tfrac{1}{3}$. 
 \end{lemma}

Next, we establish an estimate for
$L_\tau \colon X_\delta^2 \to X_\delta^0$. Let \(\varrho \in C^\infty(\R, [0,1])\) be a smooth cut-off function with
\[
\varrho(r) = \left\{ \begin{array}{ll} 1 & \textrm{for }  r < 1/3 \\
 0 & \textrm{for } r > 1/2.\end{array} \right.
\]
Given a function \(h \colon \overline{\slab_\delta} \to \R\) we define its (odd) extension \(Eh \colon \mathbb{R}^2 \to \mathbb{R}\) by 
\begin{equation}\label{eq:extension}
(Eh)(x) := \left\{ \begin{array}{ll} h(x) & \qquad \textrm{for } x \in  \slab_\delta \\
 -h(x_1, \textstyle{\pm \frac{2}{\delta}} - x_2) \varrho((|x_2| - \textstyle{\frac{1}{\delta}})\delta) &\qquad \textrm{for } \pm x_2 \geq 1/\delta. 
\end{array} \right.
\end{equation}
Notice that
\be |h|_{L^2(\slab_\delta)} \leq |Eh|_{L^2(\mathbb{R}^2)} \leq 3|h|_{L^2(\slab_\delta)},\label{compare u and Eu} \ee
and hence  $h \in L^2(\slab_\delta)$ if and only if $Eh \in L^2(\mathbb{R}^2)$.  By a standard property of odd extensions, we have in fact that $h \in X_\delta^2$ if and only if $Eh \in X$.  

Now,  let 
\be \label{E:U2perp} 
U_2^\perp := \left\{ u \in X_\delta^0 : \left( u, \, U_2 \right)_{L^2(\slab_\delta)} = 0 \right\}.
\ee
For any $|\tau| \le \frac{1}{3}$ and $u \in U_2^\perp$, we see that $Eu \in Y$ and 
\begin{align*}
\int_{\R^2} (E u) \partial_{x_2} U \,\diff x &=  \int_{\slab_\delta^c} (E u)  \partial_{x_2} U \,\diff x + \int_{\slab_\delta} u  (U_2 + (\partial_{x_2} U)_\mathrm {bc} ) \,\diff x\\
&= -\sum_{\pm} \int_{\{\frac{1}{\delta} < \pm x_2 < \frac{3}{2\delta}\}} u(x_1, \textstyle{\pm \frac{2}{\delta}} - x_2) \varrho(\delta |x_2|- 1)  \partial_{x_2} U \,\diff x 
 \\
 & \qquad +  \int_{\slab_\delta} u  (\partial_{x_2} U)_\bc  \,\diff x.
\end{align*}
Using the $L^2$ bounds of $\p_{x_2} U$ and $\Uybc$ given in Corollary~\ref{C:Unorms} and~ \ref{C:Ubc},  we estimate that
\be 
\left|\int_{\mathbb{R}^2} (Eu) \partial_{x_2} U \, \diff x \right| \lesssim 
\delta^{\frac 14} e^{-\frac{1-|\tau|}{\delta}} |u|_{L^2(\slab_\delta)} \qquad \textrm{for all } u \in U_2^\perp,
\label{U2 near orthogonality} \ee
uniformly for $|\tau| \le \tfrac{1}{3}$ and all small values of \(\delta\).  In other words, the extension $Eu$ is nearly orthogonal to $\partial_{x_2} U$ in $L_{\mathrm{e}}^2(\mathbb{R}^2)$. 
Combining these observation leads to the bound
\be  \label{P estimate} 
| P Eu |_{L^2(\mathbb{R}^2)} = \frac{\left| \int_{\mathbb{R}^2} (Eu) \partial_{x_2} U \, \diff x \right|}{ |\partial_{x_2} U|_{L^2(\mathbb{R}^2)}} \lesssim 
\delta^{\frac 14} e^{-\frac{1-|\tau|}{\delta}}  |u|_{L^2(\slab_\delta)},
\ee 
which holds for all $u \in U_2^\perp$.

\begin{lemma}[Non-degeneracy in $\slab_\delta$] \label{non-degeneracy proposition} \hfill
\begin{enumerate}[label=\rm(\alph*)] 
\item   There exists $\delta_0 >0$ and $\lambda_0 > 0$ 
such that, for all $\delta \in (0,\delta_0)$ and $|\tau| \le \tfrac{1}{3}$,
\be \label{eq:nondegeneracy}
|L_\tau u|_{L^2(\slab_\delta)} \geq \lambda_0 |u|_{L^2(\slab_\delta)}, \qquad \textrm{for all } u \in X_\delta^2 \cap U_2^\perp.
\ee 
\item \label{non-degeneracy general condition} For every $\theta \in (0,1)$, there exists $\delta_0 = \delta_0(\theta)> 0$ and $\mu_0 = \mu_0(\theta) > 0$ such that, for all $\delta \in (0,\delta_0)$, $|\tau| \le  \tfrac{1}{3}$, and $u \in X_\delta^2$ satisfying
\be |P Eu|_{L^2(\mathbb{R}^2)} \leq  \theta | u|_{L^2(\slab_\delta)},\label{near orthogonal hypothesis} \ee
we have 
\be | L_\tau u |_{L^2(\slab_\delta)} \geq \mu_0 |u|_{L^2(\slab_\delta)}.\label{general nondegeneracy ineq} \ee
\end{enumerate}
\end{lemma}

\begin{proof}  First observe that in light of \eqref{P estimate}, for any fixed $\theta \in (0, 1)$, any element $u \in U_2^\perp$ will satisfy \eqref{near orthogonal hypothesis} for $\delta$ sufficiently small.  It therefore suffices to prove part (b).  Fix $\theta \in (0,1)$ as above and let $u$ satisfy the near orthogonality condition \eqref{near orthogonal hypothesis} be given.  By linearity, we can assume without loss of generality  that $|u|_{L^2(\slab_\delta)} \leq 1$.  

From the definition  \eqref{eq:extension} of the extension $E$, 
\[ 
\left[-\Delta + \gamma^\prime(U), \, E\right] u = 0 \qquad \textrm{on } \slab_\delta \cup (\slab_{3\delta/2})^c.
\] 
We compute that, for \(\pm x_2 > 1/\delta\), one has
\begin{align*}
&\left( (- \Delta + \gamma^\prime(U)) E u\right) (x) \\ 
 & \qquad = \Delta u(x_1, \pm\textstyle{\frac{2}{\delta}} - x_2) \varrho( \pm \delta x_2 - 1 
) \mp 2 \delta \varrho^\prime(\pm\delta x_2 -1 
) \partial_{x_2} u (x_1, \pm \frac{2}{\delta} -x_2)\\ 
& \qquad \qquad + \delta^2 u(x_1, \pm\textstyle{\frac{2}{\delta}} - x_2) \varrho^{\prime\prime}(\pm\delta x_2 - 1
)  - \gamma^\prime(U) u(x_1, \pm\textstyle{\frac{2}{\delta}} - x_2) \varrho(\pm\delta x_2 - 1
).
\end{align*}
This leads directly to the following expression for the commutator on the set $\{\pm x_2 > 1/\delta\}$
\begin{equation}\label{eq:EL-commutator}
\begin{split}
\left(\left[ -\Delta + \gamma^\prime(U) , \, E \right] u\right) (x)  &  = \left( \left( -\Delta + \gamma^\prime(U) \right) Eu\right)(x) - \left( E \left(-\Delta + \gamma^\prime(U) \right) u\right)(x) \\
&= \mp 2 \delta \varrho^\prime(\pm\delta x_2 - 1
) \partial_{x_2} u (x_1, \pm\textstyle{\frac{2}{\delta}} -x_2) 
\\& \quad 
+ \delta^2 u(x_1,\pm \textstyle{\frac{2}{\delta}} - x_2) \varrho^{\prime\prime}(\pm\delta x_2 - 1 
)\\ 
&\quad - \big(\gamma^\prime(U) - \gamma^\prime(U(x_1,\textstyle{\pm\frac{2}{\delta}} - x_2)) \big) u(x_1, \textstyle{\pm\frac{2}{\delta}} - x_2) \varrho(\pm\delta x_2 - 1
).
\end{split}
\end{equation}
Now, measuring the left- and right-hand sides of \eqref{eq:EL-commutator} in $L^2(\mathbb{R}^2)$, taking into account the estimates of $\gamma^\prime(U) = 1 + O(U)$ for small $U$ and Proposition~\ref{prop:sign of Uy}, we find that 
\[
| [-\Delta + \gamma^\prime(U), \, E] u |_{L^2(\mathbb{R}^2)} \lesssim \delta |\partial_{x_2} u |_{L^2(\slab_\delta)} + \delta^2 | u | _{L^2(\slab_\delta)}.
\]
The $\partial_{x_2} u$ term above can be eliminated via interpolation:
\begin{align*}
|\partial_{x_2} u |_{L^2(\slab_\delta)}^2 & \lesssim | \Delta u|_{L^2(\slab_\delta)} |u|_{L^2(\slab_\delta)}   \lesssim |L_\tau u|_{L^2(\slab_\delta)} |u|_{L^2(\slab_\delta)} + |u|_{L^2(\slab_\delta)}^2 \\
&  \lesssim |L_\tau u|_{L^2(\slab_\delta)}^2+ |u|_{L^2(\slab_\delta)}^2.
\end{align*}
Inserting this inequality into \eqref{eq:EL-commutator}, we arrive at the commutator bound
\be | [-\Delta+ \gamma^\prime(U), \, E] u|_{L^2(\mathbb{R}^2)} \lesssim \delta \left( | L_\tau u |_{L^2(\slab_\delta)} + |u|_{L^2(\slab_\delta)}\right),  \label{final EL-commutator} \ee
independent of $\delta$, $\tau$,  and $u$.  

We are now prepared to prove the estimate \eqref{general nondegeneracy ineq}.  From Lemma~\ref{nondegeneracy in R^2 lemma} we have that
\begin{equation}\label{eq:LEu1}
\begin{aligned}
|(-\Delta +\gamma^\prime(U))Eu|_{L^2(\R^2)}^2 &= |(-\Delta +\gamma^\prime(U))(1-P)Eu|_{L^2(\R^2)}^2 \\
&\gtrsim  |(1-P)Eu|_{L^2(\R^2)}^2 = |Eu|_{L^2(\R^2)}^2 - |P Eu|_{L^2(\R^2)}^2\\
&\ge \left(1- {\theta}^2 \right) |u|_{L^2(\slab_\delta)}^2,
\end{aligned}
\end{equation}
where the last inequality follows from  hypothesis \eqref{near orthogonal hypothesis} and \eqref{compare u and Eu}. On the other hand, together \eqref{compare u and Eu} and the commutator estimate \eqref{final EL-commutator} reveal that
\begin{equation}\label{eq:LEu2}
\begin{aligned}
|(-\Delta + \gamma^\prime(U))Eu|_{L^2(\R^2)} &\lesssim  |EL_\tau u|_{L^2(\R^2)} +  \delta \left( | L_\tau u |_{L^2(\slab_\delta)} + |u|_{L^2(\slab_\delta)}\right) \\
&\lesssim  |L_\tau u|_{L^2(\slab_\delta)} + \delta |u|_{L^2(\slab_\delta)}. 
\end{aligned}
\end{equation}
Combined, \eqref{eq:LEu1} and \eqref{eq:LEu2} imply that \eqref{general nondegeneracy ineq} holds when $\delta$ is taken sufficiently small, which completes the proof.  \end{proof}

\subsection{Construction of a near-degenerate eigenfunction}

\begin{figure}
\centering
\includegraphics[scale=1]{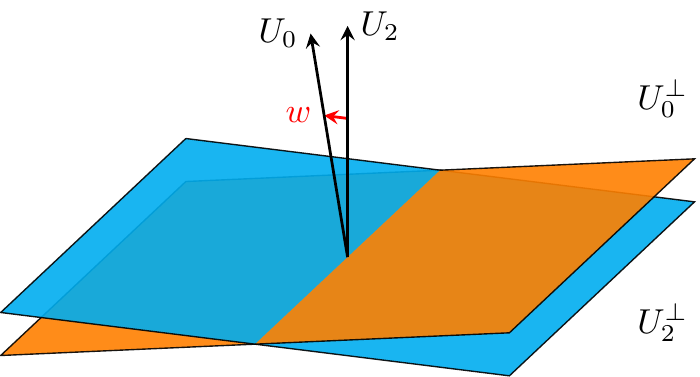}
\caption{We construct an eigenfunction $U_0$ that is almost parallel to $U_2$, and thus it will take the form $U_0 \eqsim U_2 + w$, for some $w \in U_2^\perp$ with $|w| \ll 1$.}
\end{figure}

In Section \ref{approximate eigen section}, it was shown that the function $U_2$ roughly aligns with the near-degenerate direction of $L_\tau$ in the sense that the restriction $L_\tau \colon X_\delta^2 \cap U_2^\perp \to X_\delta^0$ is uniformly positive according to \eqref{eq:nondegeneracy}, where $U_2^\perp$ is defined in \eqref{E:U2perp}.  We now refine our analysis to find a (very small)  eigenvalue $l$ and corresponding eigenfunction $U_0$ near $U_2$ that limits to $\partial_{x_2} U$ in some sense as $\delta \searrow 0$. 
Similar as \(P\) above, denote by $P_2$ the $L^2(\slab_\delta)$ orthogonal projection $X_\delta^0 \to \spn{\{U_2\}}$ and also the projection it induces from $X_\delta^2 \to \spn{\{U_2\}}$.

\begin{lemma} \label{invertibility lemma} Consider the operator 
\[ \tilde L_\tau\colon D(\tilde L_\tau) = U_2^\perp \cap X_\delta^2 = (1-P_2) X_\delta^2  
\to U_2^\perp  = (1-P_2)X_\delta^0
\]
defined by 
\[ \tilde L_\tau u =  (1-P_2) L_\tau u  \qquad \textrm{for all } u \in (1-P_2) X_\delta^2. \]
There exists $\delta_0 > 0$ such that, for all $|\tau| \le \frac 13$ and $\delta \in (0,\delta_0)$, we have that $\tilde L_\tau$ is an isomorphism and is self-adjoint as an unbounded and densely defined operator on $U_2^\perp$ with $|\tilde L_\tau^{-1}| \eqsim 1$. 
\end{lemma}

\begin{proof}
Throughout the proof, all norms and inner products are evaluated on the domain $\slab_\delta$.  By definition, $u \in D(\tilde L_\tau^*) \subset U_2^\perp$ and $\tilde u = \tilde L_\tau^* u \in U_2^\perp$ if and only if 
\[
(\tilde u, v)_{L^2} - (u, \tilde L_\tau v)_{L^2} =0 \qquad \textrm{for all } v \in U_2^\perp \cap X_\delta^2, 
\]
which holds if and only if 
\[
\left( \tilde u + \frac {(u, \, L_\tau U_2)_{L^2} }{|U_2|_{L^2}^2 } U_2, \,  aU_2 + v \right)_{L^2} - \left(u, \, L_\tau (aU_2 + v) \right)_{L^2}=0, \]
for all $a \in \R$ and  $v \in U_2^\perp\cap X_\delta^2$.   Since $L_\tau^*= L_\tau$ on $X_\delta^0$, we obtain that $u \in D(\tilde L_\tau^*)$ and $\tilde u = \tilde L_\tau^* u$ if and only if 
\[
u \in U_2^\perp \cap X_\delta^2 \quad  \text{ and } \quad \tilde u + \frac {(u, \, L_\tau U_2)_{L^2} }{|U_2|_{L^2}^2 } U_2 = L_\tau u.
\]
Thus $D(\tilde L_\tau^*) = U_2^\perp \cap X_\delta^2$, and  $\tilde u= (I-P_2) L_\tau u = \tilde L_\tau u$. This implies that $\tilde L_\tau$ is indeed self-adjoint on its domain in $U_2^\perp$. 

Next, we improve slightly the bound of $\tilde L_\tau$ in \eqref{eq:nondegeneracy}: observe that, for all \(u \in (1-P_2) X_\delta^2\),
\begin{align*}
|u|_{H^2(\slab_\delta)} &\lesssim |u|_{L^2} + |\Delta u|_{L^2} \lesssim  |u|_{L^2} + |L_\tau u|_{L^2} \lesssim  |L_\tau u|_{L^2} \lesssim  |\tilde L_\tau u|_{L^2} + |P_2 L_\tau u|_{L^2}. 
\end{align*}
But, due to equation \eqref{E:LtauU2} satisfied by $U_2$ and Lemma~\ref{L:U2}, we know that 
\[
|P_2 L_\tau u|_{L^2(\slab_\delta)} \eqsim | ( L_\tau u,\,  U_2 )_{L^2(\slab_\delta)}| = |( u, \,  L_\tau U_2 )_{L^2(\slab_\delta)}| \lesssim e^{-\frac{2(1-|\tau|)}{\delta}} |u|_{L^2(\slab_\delta)}, 
\]
and thus
\be \label{E:tL-non-deg} 
|u|_{H^2(\slab_\delta)} \lesssim |\tilde L_\tau u|_{L^2(\slab_\delta)}.
\ee
This implies that $\tilde L_\tau$ is an isomorphism from $U_2^\perp \cap X_\delta^2$ to its range --- a closed subspace of $U_2^\perp$. It follows from the self-adjointness of $\tilde L_\tau$ on $U_2^\perp$ that it is an isomorphism from $U_2 ^\perp \cap X_\delta^2$ to $U_2^\perp$. 
\end{proof}

 \begin{proposition}[Existence of $U_0$]\label{prop:U_0} 
For each $|\tau|\le  \tfrac{1}{3}$ and $ \delta > 0 $ sufficiently small, there exists an 
eigenfunction
\be \label{form of U0} 
U_0 = a_0( w + U_2) \in X_\delta^{k_0+1}, \qquad w \in (I-P_2) X_\delta^{k_0+1}, \quad |U_0|_{L^2(\slab_\delta)} =1
\ee
of \(L_\tau\) with a real eigenvalue $l = l(\delta, \tau)$:
\[ 
L_\tau U_0 = l U_0 \qquad \textrm{in } \slab_\delta.
\]
They obey the estimates  
\[
|w|_{H^{k_0+1}(\slab_\delta)} \lesssim \delta^{\frac 14} \left|\log{\delta} \right|^{\frac 12} e^{-\frac {2(1-|\tau|)}\delta}, \quad 0< l(\delta, \tau) \eqsim \delta^{\frac 12} e^{-\frac {2(1-|\tau|)}\delta}, \quad 0< a_0 \eqsim 1.
\]
Moreover, for fixed \(\delta\), \(l\) is $C^{k_0-1}$ in $\tau$ and \(U_0 \in X_\delta^{k+2} \) is $C^{k_0-k-1}$ in \(\tau\) for $0\le k\le k_0-1$, respectively. 
\end{proposition}

Here $a_0$ is simply a normalizing constant so that $|U_0|_{L^2(\slab_\delta)} =1$. 

\begin{proof}
From Lemma~\ref{invertibility lemma}, we know that there exists $\delta_0 > 0$ such that, for any $\delta \in (0,\delta_0)$, $\tilde L_\tau$ is an isomorphism from $(I-P_2) X_\delta^2$ to $(I-P_2) X_\delta^0$.  By \eqref{E:tL-non-deg}, its inverse satisfies 
 \begin{equation}\label{eq:norm of A}
 |\tilde L_\tau^{-1} |_{\mathcal{L}((I-P_2)X_\delta^0; (I-P_2)X_\delta^2)} \leq \lambda_0^{-1} \qquad \textrm{for all } |\tau| \le  \tfrac{1}{3},
 \end{equation}
 for some $\lambda_0>0$ independent of $|\tau|\le \frac 13$ and small $\delta>0$. A function $U_2 + w$, with $w\in(I-P_2) X_\delta^2$, is an eigenfunction corresponding to $l$ if 
 \[
 L_\tau (U_2 + w) = l (U_2+w). 
 \] 
Taking the inner product of the above equation with $U_2$ yields 
\[
l = \frac{(w+U_2, \, L_\tau U_2)_{L^2(\slab_\delta)}}{ |U_2|_{L^2(\slab_\delta)}^2}.
\]
On the other hand, applying $I-P_2$ to the eigenfunction equation and recalling that $P_2 w=0$, we see that 
\begin{equation} \label{E:w}
\tilde L_\tau w = lw - (I-P_2) L_\tau U_2.
\end{equation}
This motivates us to consider the mapping $\Lambda\colon U_2^\perp \to U_2^\perp$ defined by 
\be \label{def contraction Lambda} 
\Lambda(w) = \ell(w) \tilde L_\tau^{-1} w - \tilde L_\tau^{-1} (1-P_2) L_\tau U_2, 
\ee
where
\be \label{def contraction ell} 
\ell(w) = \frac{(w+U_2, \, L_\tau U_2)_{L^2(\slab_\delta)}}{ |U_2|_{L^2(\slab_\delta)}^2}
\ee
is the presumptive eigenvalue. Clearly a small fixed point $w\in U_2^\perp$ of $\Lambda$ yields an eigenfunction $w+U_2$ of $L_\tau$ close to $U_2$ associated to the eigenvalue $\ell(w)$. 

It is straightforward to estimate  
 \begin{equation}\label{eq:l(w) explicit} \begin{split}
 |\ell(w)| \lesssim |L_\tau U_2|_{L^2(\slab_\delta)} |w|_{L^2(\slab_\delta)} + (U_2, L_\tau U_2)_{L^2(\slab_\delta)},
\end{split} \end{equation}
and 
\begin{align*}
|\ell(w_1) - \ell(w_2)| & \leq  | U_2|_{L^2(\slab_\delta)}^{-2}  |L_\tau U_2|_{L^2(\slab_\delta)} |w_1 - w_2|_{L^2(\slab_\delta)}.
\end{align*}
Likewise, we have 
 \begin{align*} |\Lambda(w_1) - \Lambda(w_2)|_{L^2(\slab_\delta)} & \leq |\ell(w_1)| |\tilde L_\tau^{-1}(w_1 - w_2)|_{L^2(\slab_\delta)} + |\ell(w_1) - \ell(w_2)| |\tilde L_\tau^{-1} w_2|_{L^2(\slab_\delta)} 
 \end{align*}
and 
 \[ 
| \Lambda(0)|_{L^2(\slab_\delta)} \lesssim |L_\tau U_2|_{L^2(\slab_\delta)},
 \]
where all above inequalities are uniform in $|\tau|\le \frac 13$ and small $\delta$. 
Consequently, Lemma \ref{L:U2} and \eqref{eq:norm of A} imply 
\(\Lambda\) is a contraction map that sends \(B_1\), the closed unit ball centered at the origin in $X_\delta^0$, to itself.  It therefore has a 
 unique fixed point $w^* = w^*(\delta, \tau) \in (1-P_2) X_\delta^2 \cap B_1$. This yields the eigenvalue $l = \ell(w^*)$  defined by \eqref{def contraction ell} and the corresponding eigenfunction $w+ U_2$ whose  higher Sobolev regularity is due to the ellipticity in \eqref{E:w}.
The normalizing constant $a_0>0$ is chosen such that $|U_0|_{L^2 (\slab_\delta)}=1$. 
Since $\gamma \in C^{k_0}$ and $U \in C^{k_0+2}$ with  exponential decay, it is easy to see that $L_\tau\colon X_\delta^{k+2} \to X_\delta^k$ is $C^{k_0-k-1}$ in $\tau$ for $k\ge 0$. From standard spectral theory, the simple eigenvalue $\ell$  is $C^{k_0-1}$ in $\tau$ and the {\it unit} eigenfunction $U_0 \in X_\delta^{k+2}$ of $L_\tau$ is $C^{k_0-k-1}$ in $\tau$ for $0\le k\le k_0-1$. 
As $w^*$ is a fixed point of the contraction $\Lambda$, its definition \eqref{def contraction Lambda} and Lemma \ref{L:U2}  imply 
\[
|w^*|_{L^2(\slab_\delta)} \lesssim |\Lambda(0)|_{L^2(\slab_\delta)} \lesssim \delta^{\frac 14} \left|\log{\delta}\right|^{\frac 12} e^{-\frac {2(1-|\tau|)}\delta}. 
\]
The higher Sobolev norms satisfy similar estimates due to the elliptic regularity given in \eqref{E:w}. Finally, we conclude from \eqref{def contraction ell}, the above inequality, and Lemma \ref{L:U2}, that 
\[
\left|\ell (w^*) -  \frac{(U_2, \, L_\tau U_2)_{L^2(\slab_\delta)}}{ |U_2|_{L^2(\slab_\delta)}^2}  \right| \le e^{-\frac {4(1-|\tau|)}\delta}.
\]
Along with Lemma \ref{L:U2}, this yield the desired estimate on $\ell$. The positivity of $\ell$ is a consequence of the sign of $(U_2, \, L_\tau U_2)_{L^2(\slab_\delta)}$ proved in Lemma \ref{L:U2}. 
\end{proof}

Using the estimates just obtained, we can now confirm that $L_\tau$ is invertible (with near-degeneracy in the $U_0$ direction) and, more important, that the inverse of its restriction to the orthogonal complement of $U_0$ is \emph{bounded independently} of $\delta$. That said, let $U_0$ be given as in Proposition \ref{prop:U_0}
and denote its orthogonal complement in $X_\delta^0$ 
by $U_0^\perp$.  

\begin{lemma}[Invertibility of $L_\tau$] \label{invertibility L on U_0 perp} 
There exists $\delta_0 > 0$ such that, for all $\delta \in (0,\delta_0)$ and $|\tau| \le \tfrac{1}{3}$, 
\be L_\tau\colon X_\delta^2 \to X_\delta^0 \qquad \textrm{is invertible}.\label{Lsigma invertible} \ee
Moreover, there exists $\mu_0 = \mu_0(\delta_0) > 0$ such that 
\[
 | L_\tau u |_{L^2(\slab_\delta)} \geq \mu_0 | u |_{L^2(\slab_\delta)} \qquad \textrm{for all } u \in 
U_0^\perp \cap X_\delta^2
\]
where $U_0^\perp$ is the $L^2(\slab_\delta)$ complement of $U_0$ in $X_\delta^0$. 
\end{lemma}

\begin{proof}  
Since $L_\tau$ is self-adjoint and $U_0$ is an eigenfunction of the eigenvalue $l$, it is standard that $U_0^\perp$ is invariant under $L_\tau$ in the sense that 
\[ L_\tau (U_0^\perp \cap X_\delta^2) \subset U_0^\perp.\]
Because of $l>0$ and again the self-adjointness of $L_\tau$, it suffices to prove that $L_\tau|_{U_0^\perp \cap X_\delta^2}: U_0^\perp \cap X_\delta^2 \to U_0^\perp$ has a lower bound 
 independent of $|\tau| \le \frac 13$ and small $\delta>0$. In fact, recall 
\[ U_0 = a_0 (U_2 + w), \qquad \textrm{with} \quad  w\in U_2^\perp, \quad |w|_{L^2(\slab_\delta)} \lesssim e^{-\frac {2(1-|\tau|)}\delta}.\]
Any $v \in U_0^\perp$ can be written as 
\[
v = v_1 + b U_2, \; \text{ where } \; v_1= (I-P_2) v \in U_2^\perp.
\]
We have 
\[
(v_1, w)_{L^2(\slab_\delta)} =(v, w)_{L^2(\slab_\delta)} = (v, \frac {U_0}{a_0} - U_2)_{L^2(\slab_\delta)} = - (v, U_2)_{L^2(\slab_\delta)}= -b |U_2|_{L^2(\slab_\delta)}^2
\]
and thus 
\[
b 
= - \frac {(v_1, w)_{L^2(\slab_\delta)}} {|U_2|_{L^2(\slab_\delta)}^2} 
= - \frac {(v, w)_{L^2(\slab_\delta)}} {|U_2|_{L^2(\slab_\delta)}^2}.
\]
It implies that $U_0^\perp$ and $U_2^\perp$ are isomorphic through 
\[
v_1 = (I-P_2) v \text{ with } |v - v_1|_{L^2(\slab_\delta)} \eqsim |b| \lesssim e^{-\frac {2(1-|\tau|)}\delta} |v|_{L^2(\slab_\delta)},
\]
where $ a_0 \eqsim 1$ was also used. Together with Lemmas \ref{L:U2} and  \ref{non-degeneracy proposition} we obtain
\[
|L_\tau v|_{L^2(\slab_\delta)} \ge |L_\tau v_1|_{L^2(\slab_\delta)} - |b| |L_\tau U_2|_{L^2(\slab_\delta)} \ge \frac {\lambda_0}2 |v_1|_{L^2(\slab_\delta)} \ge \frac {\lambda_0}4  |v|_{L^2(\slab_\delta)}
\]
which completes the proof. 
\end{proof}

The invertibility of $L_\tau$ also holds in higher Sobolev spaces due to  elliptic theory.

\begin{corollary} \label{C:invertibility}
There exists $\delta_0 > 0$ such that, for all $\delta \in (0,\delta_0)$, $|\tau| \le \tfrac{1}{3}$, and $0\le k \le k_0-1$, 
\be 
L_\tau\colon X_\delta^{k+2} \to X_\delta^k
\qquad \textrm{is invertible}.\label{Lsigma invertible-k} 
\ee
Moreover, there exists $\mu_0 = \mu_0(\delta_0) > 0$ such that 
\be 
| L_\tau u |_{H^k(\slab_\delta)} \geq \mu_0 | u |_{H^{k+2}(\slab_\delta)} \qquad \textrm{for all } u \in 
U_0^\perp \cap X_\delta^{k+2}.
\label{U0perp nondegeneracy estimate} \ee
\end{corollary}

\section{Proof of the main result} \label{proof main theorem sec}

In this section we complete the argument leading to the proof of Theorem~\ref{main euler theorem}.

\subsection{Normal bundle coordinates} \label{tubular section}

Recall from Section~\ref{sec:perturbation} that the waves we study are represented by two quantities: the boundary value  $\Gammas$ of a conformal mapping, that determines the fluid domain, and a (rescaled) stationary stream function $\varphi$ that gives the velocity field.  Our basic approach is to construct waves for which $|\Gammas|_{H^{k_0}} \ll 1$  and $\varphi$ is a perturbation of $U(\tau) - U(\tau)_\bc$, where the parameter $\tau \in (-\frac 13, \frac 13)$ selects the approximate altitude of the center of vorticity.  

At this stage, we have obtained detailed information regarding the spectrum of the linearized operator 
\[ L_\tau= -\Delta +\gamma^\prime\left(U(\tau)\right)\colon X_\delta^{k+2} \to X_\delta^{k}\]
and its dependence on $\tau$ and $\delta$.  In particular, we proved in Proposition~\ref{prop:U_0} that there exists a unique simple eigenvalue $l = l(\delta, \tau)$, associated to an eigenfunction $U_0(\tau)$, that converges to $0$ exponentially fast as $\delta \searrow 0$.  This presents an obvious obstruction to a na\"ive fixed point scheme.  We will see that $\tau$ is the key to ameliorating the issue.

To see the connection, observe that the family 
\[ 
\mathcal{C} := \left\{ U(\tau) - U(\tau)_{\bc} : \tau \in (-\tfrac{1}{3},\tfrac{1}{3}) \right\} 
\]
can be viewed as a $C^{k_0+2-k}$ curve in the ambient space $X_\delta^k$.    At a fixed $\tau$, the tangent vector to $\mathcal{C}$ is 
\[
T_\tau \mathcal{C}= \p_\tau \left( U(\tau) - U(\tau)_\bc\right) = - \delta^{-1} \left( \p_{x_2} U (\tau) - \left( \p_{x_2} U (\tau) \right)_\bc \right) = - \delta^{-1} U_2(\tau) \sim - \delta^{-1} U_0(\tau), 
\]
where the second equality follows from the linearity of the boundary correction operator. The above calculation shows that the tangent direction along the curve $\mathcal{C}$ is almost parallel to the near-degenerate subspace.

Therefore, our strategy is to seek a (rescaled) stationary stream function of the form 
\be \label{E:ansatz}
\varphi = U(\tau) - U(\tau)_\bc + v, 
\ee
with the unknowns 
\be \label{E:bundle}
(\tau, v) \in \vecbundlek := \left\{ (\tau, v) : \tau \in (-\tfrac{1}{3},\tfrac{1}{3}), \; v \in X_\delta^k \cap U_0(\tau)^\perp \right\}, \quad k\ge 2.
\ee
This ensures that $v$ avoids the near-degenerate direction of $L_\tau$. While the linear part $g - \alpha^2 \Diff^2$ of the Bernoulli boundary condition \eqref{perturbed Gamma eq} is already invertible.  We may then perform a Lyapunov--Schmidt reduction:  for each fixed $\tau$, we solve for $v$ and $\Gammas$, leaving a one-dimensional problem of the form $b(\tau) = 0$,  for a certain bifurcation function $b$.   Finally, we will appeal to an intermediate value theorem argument to infer the existence of solutions to this reduced problem, as anticipated by the model calculation carried out in Section~\ref{sec:formulation}. 

\begin{figure}
\centering
\includegraphics[scale=1.00]{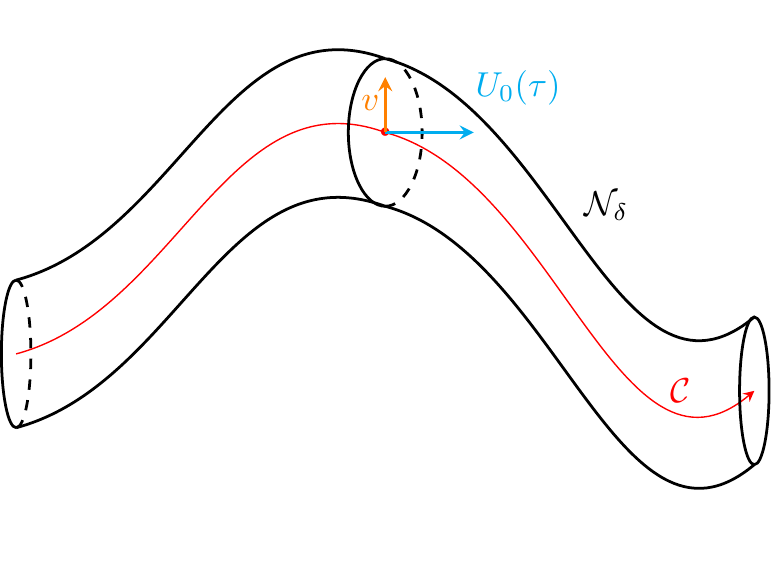}
\caption{Schematic of the tubular neighborhood $\mathcal{N}_\delta$ of the curve $\mathcal{C}_\delta$.  At each $\tau \in (-\tfrac{1}{3}, \tfrac{1}{3})$, we locally decompose the space $X_\delta^k$ into a component $v$ in the non-degenerate direction $U_0(\tau)^\perp$, and a component in the near-degenerate direction $U_0(\tau)$.}
\label{fig:tubular}
\end{figure}

It is therefore imperative that the Lypanuov--Schmidt reduction be performed in such a way that $b(\tau)$ is \emph{continuous} (or even smooth).  Because the near-degenerate and non-degenerate subspaces vary as we change $\tau$, it is natural to view $\vecbundlek$ as a smooth vector bundle over the base $(-\tfrac{1}{3},\tfrac{1}{3})$, with the fibers being the non-degenerate subspaces 
\[ 
\vecbundlek^\tau := X_\delta^k \cap U_0(\tau)^\perp,
\]
 see also Figure~\ref{fig:tubular}.
According to Proposition \ref{prop:U_0}, the $C^{k_0}$-regularity of $\gamma$ ensures that $\tau \mapsto U_0(\tau) \in X_\delta^{k+2}$ is $C^{k_0-k-1}$ for $0 \leq k \leq k_0-1$, and hence the orthogonal projection $P_0(\tau)$ onto $\spn{U_0(\tau)}$ enjoys the same regularity with respect to $\tau$. It then follows that each $\tau_0 \in (-\tfrac{1}{3},\tfrac{1}{3})$ is contained in a neighborhood $\mathcal{I}_0$ such that the mapping 
\[ (\tau, v) \in  \mathcal{I}_0 \times \vecbundlek^{\tau_0} \mapsto \left(\tau, (I-P_0(\tau)) v \right)\in  \vecbundlek, \qquad \textrm{for } 2 \leq k \leq k_0 + 1\]
is a $C^{k_0-k+1}$ local trivialization of $\vecbundlek$.   Note that here and in the sequel, we reserve cursive script for bundles. In the Lyapunov--Schmidt reduction, we fix $\tau$, while tracking the continuous dependence on it. 

\begin{remark}
While continuity in $\tau$ is sufficient for our purpose, in differential geometry, there are standard notions of smoothness of mappings related to vector bundles based on the smoothness of the  trivializations, which allow  implicit function theorem type arguments to be carried out as on flat spaces or manifolds. 
Moreover, it is standard to prove that 
\[ 
\chi \colon (\tau, v) \in \vecbundlek \mapsto  v + U(\tau) - U(\tau)_\bc \in X_\delta^k, \qquad \textrm{for } 2 \leq k \leq k_0, 
\]
defines a $C^{k_0-k+1}$ local coordinate map (usually referred to as the transversal bundle coordinates) near $\mathcal{C}$. 
\end{remark}

To simplify notation, we introduce the set   
\[ \vecbundlevGk := 
\vecbundlek \times H_{\mathrm{e}}^k(\mathbb{R}) 
\]
and endow it with the structure of a vector bundle over $(-\tfrac{1}{3},\tfrac{1}{3})$ having fibers 
\[ \vecbundlevGk^\tau := \vecbundlek^\tau \times H_{\mathrm{e}}^k(\mathbb{R}),\]
and locally trivialized in the obvious way.  

\subsection{Lyapunov--Schmidt reduction} \label{reduction section}

Let us now reconsider the elliptic system \eqref{perturbed problem},
\[ 
\begin{cases} 
(-\Delta + \gamma^\prime(U)) v + F(\tau, v, \Gammas) &= 0 \qquad \textrm{in } \slab_\delta, \\
 (g - \alpha^2 \Diff^2) \Gammas + G(\tau, v, \Gammas) &= 0 \qquad \textrm{on } \mathbb{R},
 \end{cases}
 \]
from this geometrical standpoint.  In the previous subsection, we argued that this system is equivalent to finding $\Gammas$ together with a scaled stream function having the ansatz 
\[ \varphi = v+U(\tau) - U(\tau)_\bc, \qquad  (\tau,v) \in \vecbundlek.\]
As before, we suppress the dependence of $U$ and $U_\bc$ on $\tau$ whenever there is no risk of confusion.  With a slight abuse of notation we also as above view 
 \begin{align} F(\tau, v, \Gammas) & = |1+\Gamma^\prime(\delta \placeholder )|^2 \gamma(v+U- U_\bc) - \gamma(U) - \gamma^\prime(U)v + U_\bc   \label{new def F}, \\
G(\tau, v, \Gammas)  &=
\frac{1}{2 \delta^{2}} A(\Gammas)^{-1} \left[ \frac{\left(\p_{x_2} (v + U - U_\bc)(\placeholder,\frac{1}{\delta}) \right)^2}{(1 +|\Diff| \coth{(2|\Diff|)} \Gammas)^2 + {\Gammas^\prime}^2}    \right], \label{new def G} \end{align}
from \eqref{def F} and \eqref{def G} to be the bundle map from a subset (with $\Gammas$ small) of $\vecbundlevGk$ to $\mathscr{W}_{\delta, k-2}$.
It is easily seen that the slightly reinterpreted $(F,G)$ enjoy the same regularity as in Lemma~\ref{equivalence lemma}.

Projecting the semilinear elliptic problem into the near-degenerate and non-degenerate subspaces (which are invariant under $L_\tau$), we can reconfigure the governing equations as the following system:
\bse \label{contraction map}
\begin{align}
P_0 F(\tau, v, \Gammas) &= 0  \quad \text{ in } \slab_\delta, \label{contraction map U0 span} \\
(-\Delta +\gamma^\prime(U)) v +  (I-P_0)  F(\tau, v, \Gammas)  &= 0 \quad \text{ in } \slab_\delta, \label{contraction map U0 perp} \\
(g   - \alpha^2  \Diff^2) \Gammas + G(\tau, v, \Gammas) &= 0 \quad \text{ on } \mathbb{R}. \label{contraction map Bernoulli} 
\end{align} \ese 

Notice that for a fixed $\tau$, \eqref{contraction map U0 perp}--\eqref{contraction map Bernoulli} are solved on the fiber $\vecbundlevGk^\tau$.    In the next lemma, we prove that one can always do this, and the solution depends smoothly on $\tau$.  We therefore reduce the system to the one-dimensional equation \eqref{contraction map U0 span} related to the near-degenerate subspace.  

\begin{lemma}[Lyapunov--Schmidt reduction] \label{reduction lemma} There exists $C, \delta_0 > 0$ such that, for all $\delta \in (0,\delta_0)$ and $\tau\in (-\frac 13, \frac 13)$, there exists a solution $(\tilde v(\tau),  \Gammast(\tau)) \in \mathscr{W}_{\delta, k_0}^\tau$ to \eqref{contraction map U0 perp}--\eqref{contraction map Bernoulli} which is unique in the set 
\[
\left\{ (v, \Gammas) \in \mathscr{W}_{\delta, k_0}^\tau \ : \ |v|_{H^{k_0} (\slab_\delta)} + C \delta^{-1}  |\Gammas|_{H^{k_0} (\R)} \le \delta^{k_0+1} \right\},
\]
and satisfies 
\[
|\tilde v|_{H^{k_0} (\slab_\delta)} + C \delta^{-1} | \Gammast|_{H^{k_0} (\R)} \lesssim C  \delta^{ -k_0} e^{-\frac {2(1-|\tau|)}\delta}, \quad |\Gammast - \eta_0|_{H^{k_0} (\R)} \lesssim C \delta^{\frac 34 -2k_0} e^{-\frac {3(1-|\tau|)}\delta}, 
\]
where 
\[
\eta_0 = - 2\delta^{-2} (g-\alpha^2 \Diff^2)^{-1} \left( \left( \p_{x_2} U(\tfrac \placeholder \delta, \tfrac 1\delta)\right)^2 \right).
\]
Moreover, $(\tilde v, \Gammast) \in H^{k_0} (\slab_\delta) \times H^{k_0} (\R)$ depends continuously on $\tau$.  
\end{lemma}

\begin{remark} As a consequence, the system \eqref{contraction map} is locally equivalent to the one-dimensional problem 
\be \label{reduced problem} 
0 = b(\tau) := \left( U_0(\tau), \,  F(\tau, \tilde v(\tau), \, \Gammast(\tau)) \right)_{L^2(\slab_\delta)} =  \left( U_0, \,  L_\tau \tilde v + F(\tau, \tilde v, \, \Gammast) \right)_{L^2(\slab_\delta)}. \ee
Also, it is worth noting that, since $\gamma \in C^k$ for any $2\le k \le k_0$, the above lemma holds for all such $k$. The uniqueness property of $(\tilde v(\tau), \Gammast(\tau) )$ implies that it is independent of $k$. 
\end{remark}

\begin{proof}
Let $\delta \in (0,\delta_0)$ be given, where $\delta_0$ will determined over the course of the proof, which is largely based on the estimates given in Lemma \ref{equivalence lemma}. To tame the singular bound $\delta^{-1}$ of $D_{\Gammas} F$, 
we introduce the rescaled variable
\[  
\Gammasc := \frac{C}{\delta} \Gammas, 
\]
where $C>0$ will be determined independent of $\tau$ and $\delta$, and the corresponding scaling of the nonlinearities
\[ 
\check F(\tau, v, \Gammasc) := F(\tau, v, \frac \delta C \Gammasc), \qquad \check G(\tau, v, \Gammasc) := \frac{C}{\delta} G(\tau, v, \frac \delta C \Gammasc). 
\]

Denote by 
\begin{align*}
&L_1^{-1}(\tau) =\left( \left(-\Delta+\gamma^\prime(U(\tau))\right)|_{\mathscr{X}_{\delta, k_0}^\tau }\right)^{-1} \colon \mathscr{X}_{\delta, k_0-2}^\tau \to \mathscr{X}_{\delta, k_0}^\tau, \\
&L_2^{-1} = (g-\alpha^2 \Diff^2)^{-1}\colon H^{k_0-2} (\mathbb{R}) \to H^{k_0}(\mathbb{R}),
\end{align*}
where we recall that the existence and boundedness of $L_1^{-1}(\tau)$ were established in Corollary 
\ref{C:invertibility}. 
In particular, notice that, because $-\Delta + \gamma^\prime(U)\colon  \mathscr{X}_{\delta, k_0}^\tau \to \mathscr{X}_{\delta, k_0-2}^\tau$ is self-adjoint with respect to the $L^2 (\slab_\delta)$ inner product and $U_0$ is an eigenfunction, the range of $L_1^{-1}(\tau)$ is contained in $U_0(\tau)^\perp$.   
Then we see that $(\tau, v, \Gammas)$ solve \eqref{contraction map U0 perp} and \eqref{contraction map Bernoulli} if and only if $(\tau, v, \Gammasc)$ is a fixed point of the mapping 
\[ 
\Lambda^\tau = \left( \Lambda_1^\tau(v, \Gammasc), \Lambda_2^\tau(v, \Gammasc)\right) \colon B \to  \mathscr{W}_{\delta, k_0}^\tau  
\]
given by
\be \begin{split} \Lambda_1^\tau(v, \Gammasc) &  =  - L_1^{-1}(\tau) (I-P_0(\tau)) \check F(\tau, v,  \Gammasc), \\ 
\Lambda_2^\tau(v, \Gammasc) & =  - L_2^{-1} \check G( \tau, v,  \Gammasc) \end{split} \label{def reduction Lambda} \ee
on  the set 
\[
B = \{ (v,  \Gammasc) \in \mathscr{W}_{\delta, k_0}^\tau \ : \ |v|_{H^{k_0} (\slab_\delta)} + | \Gammasc|_{H^{k_0} (\R)} \le \delta^{k_0+1}\}. 
\]

From Lemma \ref{equivalence lemma}, we have 
\[
|\Lambda^\tau (0, 0)|_{H^{k_0}(\slab_\delta) \times H^{k_0} (\R)} \lesssim C \delta^{ -k_0} e^{-\frac {2(1-|\tau|)}\delta}, \quad |D \Lambda^\tau|_{C^0(B, \mathcal{L}( \mathscr{W}_{\delta, k_0}^\tau))} \lesssim C^{-1}.
\]
Therefore, for a sufficiently large $C>0$, which can be chosen independently of $\delta$ and $\tau$,  $\Lambda^\tau$ is a contraction on $B$, and so it possesses a unique fixed point 
\[ (\tilde v(\tau), \Gammasc(\tau)) = (\tilde v(\tau), C \delta^{-1}  \Gammast(\tau)) \in B.\]
Moreover, we have the estimate
\[
|\tilde v(\tau)|_{H^{k_0} (\slab_\delta)} + |\Gammasc(\tau) |_{H^{k_0} (\R)} \lesssim |\Lambda^\tau (0, 0)|_{H^{k_0}(\slab_\delta) \times H^{k_0} (\R)} \lesssim C \delta^{ -k_0} e^{-\frac {2(1-|\tau|)}\delta}.
\]
The continuity of $\tilde v(\tau)$ and $\Gammast(\tau)$ follows from the continuity of $\Lambda^\tau$ in $\tau$, where we can view it as a mapping defined on a smooth bundle. 

Finally, we identify the leading order term of $\Gammast (\tau)$. Due to the fixed point property, we have 
\[
(g-\alpha^2 \Diff^2)  \Gammast (\tau)  = - G(\tau, \tilde v, \Gammast).
\]
Lemma \ref{equivalence lemma} and the above upper bounds of $(\tilde v, \Gammasc)$  imply
\[
|G(\tau, \tilde v, \Gammast) - G(\tau, 0, 0)|_{H^{k_0-2} (\R)} \lesssim C \delta^{\frac 34 -2k_0} e^{-\frac {3(1-|\tau|)}\delta}. 
\]
From \eqref{E:G0}, \eqref{E:temp-a}, \eqref{E:temp-0}, and the scaling property, we have 
\[
\left|G(\tau, 0, 0) - 2\delta^{-2} \left( \p_{x_2} U(\tfrac \placeholder \delta, \tfrac 1\delta)\right)^2 \right|_{H^{k_0-2} (\R)} \lesssim  \delta^{-\frac 12 -k_0} e^{-\frac {3(1-|\tau|)}\delta}
\]
which along with the above inequality yields the desire estimate on $\Gammast(\tau)$. 
\end{proof}

\subsection{Proof of the main result} \label{main proof section}

\begin{proof}[Proof of Theorem~\ref{main euler theorem}]

The Lyapunov--Schmidt reduction carried out in Lemma~\ref{reduction lemma} shows that it suffices to find $\tau \in (-\tfrac{1}{3},\tfrac{1}{3})$ with $b(\tau) = 0$, where $b(\tau)$ is defined in \eqref{reduced problem}.  Our strategy will be to relate the bifurcation equation to the model calculation \eqref{motivational identity}. 

With that in mind, fix $\tau \in (-\frac 13,\tfrac{1}{3})$ and recall 
\[ 
(\tilde v,\Gammast) = (\tilde v(\tau), \Gammast(\tau)), \quad  U = U(\tau), \quad U_0(\tau) = a_0(U_2 + w),  \quad\textrm{and} \quad \varphi = U - U_\bc + \tilde v, 
\]
recalling that $U_0$, $U_2$, $a_0$, and $w$ were obtained in Section~\ref{sec:Spec}.  In particular, $1\eqsim a_0 = a_0(\tau)>0$ is a normalizing constant introduced to ensure that $| U_0 |_{L^2} = 1$. 

Since $(\tilde v, \Gammast)$ solves \eqref{contraction map U0 perp}, 
we have 
\be \label{E:b}
L_\tau \tilde v + F(\tau, \tilde v, \Gammast)  = b(\tau) U_0(\tau).
\ee
  Now, let 
\[ \psi (\tau) = \varphi (\tau) \circ \left(\id+ \delta^{-1} \tilde \Gamma(\delta \placeholder)\right)^{-1},
\]
where $\tilde \Gamma = \tilde \Gamma_1 + \I \tilde \Gamma_2$ is the holomorphic function constructed from  $\Gammast$ through \eqref{def Gammas}. According to Lemma~\ref{reduction lemma}, the domain of $\psi$ is the (slightly) perturbed strip 
\[
\tilde \Omega(\tau) = \left(\id+ \delta^{-1} \tilde \Gamma(\delta \placeholder)\right) (\slab_\delta) \sim \slab_\delta.
\]
For clarity, we use $y = (y_1, y_2)$ as the coordinate variable in $\tilde \Omega(\tau)$.  It is easy to compute 
\[
\p_{y_2} \psi = \left(\frac \I {1+ \tilde \Gamma'(\delta \placeholder)} \cdot \nabla  \varphi\right) \circ \left(\id+ \delta^{-1} \tilde \Gamma(\delta \placeholder)\right)^{-1},
\]
where the complex number $\I (1+ \tilde \Gamma'(\delta \placeholder))^{-1}$ is understood as a two-dimensional vector. Corollary \ref{C:Ubc}, Proposition \ref{prop:U_0}, and Lemma \ref{reduction lemma} together imply that  
\[
\left|U_0(\tau) - a_0(\tau) 
\frac \I {1+ \tilde \Gamma'(\delta \placeholder)} \cdot \nabla  \varphi \right|_{L^2(\slab_\delta)} 
\ll1 = |U_0|_{L^2(\slab_\delta)}.
\]
In view of \eqref{E:b}, we have that \eqref{reduced problem} holds for $(\tau, \tilde v, \Gammast)$ if and only if 
\[
\tilde b(\tau) := \left(  
\frac \I{1+ \tilde \Gamma'(\delta \placeholder)} \cdot \nabla  \varphi, L_\tau \tilde v + F(\tau, \tilde v, \Gammast) \right)_{L^2(\slab_\delta)} = 0.
\]

By the definitions of $F$ and the boundary correction operator, 
\[
L_\tau \tilde v + F(\tau, \tilde v, \Gammast) = -\Delta  \varphi (\tau) + |1+\tilde \Gamma' (\delta \placeholder) |^2 \gamma( \varphi),
\]
which, along with the coordinate change $y=x + \delta^{-1} \tilde \Gamma(\delta x)$, gives 
\[
\tilde b(\tau) = \int_{\tilde \Omega(\tau)} \left(-\Delta \psi + \gamma(\psi)\right) \p_{y_2} \psi \, \diff y. 
\]
Following the same calculation leading to \eqref{motivational identity}, we then find that 
\[
\tilde b(\tau) = - \frac 12 \int_{\p \tilde \Omega (\tau)} |\nabla \psi|^2 N_2 \, \diff S_y 
\]
where 
\[
N = (N_1, N_2) =\pm \left(\frac {\I+\I\tilde \Gamma'(\delta \placeholder)}{|1+\tilde \Gamma'(\delta \placeholder)|}\right) \circ \left(\id+ \delta^{-1} \tilde \Gamma(\delta \placeholder)\right)^{-1} 
\]
is the outward unit normal vector on the upper/lower component of $\p \tilde \Omega (\tau)$, and 
\[
\diff S_y = |1+\tilde \Gamma'(\delta \placeholder)| \circ \left(\id+ \delta^{-1} \tilde \Gamma(\delta \placeholder)\right)^{-1} \, \diff x_1
\]
the length element along $\p \tilde \Omega(\tau)$. We can rewrite $\tilde b(\tau)$ as an integral on $\slab_\delta$ by reversing the coordinate change:
\be \label{E:tb-1} \begin{split}
\tilde b(\tau) &= - \frac 12 \int_\R \frac {1+ \p_{x_1} \tilde \Gamma_1 (\delta x_1, 1)}{|1+\tilde \Gamma'(\delta x_1, 1)|^2} \left| \p_{x_2} \varphi (x_1, \tfrac 1\delta)\right|^2 \, \diff x_1 \\
& \quad +  \frac 12 \int_\R \frac {1+ \p_{x_1} \tilde \Gamma_1 (\delta x_1, -1)}{|1+\tilde \Gamma'(\delta x_1, -1)|^2} \left|\p_{x_2} \varphi (x_1, -\tfrac 1\delta) \right|^2 \, \diff x_1.
\end{split} \ee
Notice that tangential derivatives do not appear because $\varphi|_{\p \slab_\delta}=0$.  Without loss of generality, we just consider the first term. From the definition of $\varphi$, \eqref{E:temp-a}, \eqref{E:temp-0},
 Lemma \ref{reduction lemma} (taking $k_0 = 2$), and the trace theorem, we obtain 
\begin{align*}
\left| \p_{x_2} \varphi (\placeholder, \tfrac 1\delta) - 2 \p_{x_2} U(\placeholder, \tfrac {1-\tau}\delta) \right|_{L^2 (\R)} \lesssim |\tilde v|_{H^2 (\slab_\delta)} + \delta^{\frac 34} e^{-\frac {2(1-|\tau|)}\delta} \lesssim  \delta^{-2} e^{-\frac {2(1-|\tau|)}\delta},
\end{align*}
and 
\[
\left| \p_{x_2} \varphi (\placeholder, \tfrac 1\delta)\right|_{L^2 (\R)} \lesssim \delta^{\frac 14} e^{-\frac {1-|\tau|}\delta}.
\]
Therefore, again Lemma \ref{reduction lemma} implies 
\[
|\tilde b(\tau) - \tilde b_1(\tau)  | \lesssim  \delta^{-\frac 74} e^{-\frac {3(1-|\tau|)}\delta}, 
\]
where 
\[
\tilde b_1(\tau) := - 2 \int_\R \left(\p_{x_2} U  (x_1, \tfrac {1-\tau}\delta)\right)^2 - \left(\p_{x_2} U(x_1, \tfrac {1+\tau}\delta)\right)^2 \, \diff x_1.
\]
Here we have used the radial symmetry of $U$ to slightly simplify the expression.  Clearly, it also implies that $\tilde b_1$ is odd. 

Due to the exponential localization, $\tilde b_1$ can be effectively determined by integrating only over a $\delta$-dependent but compact interval.  Indeed, from Proposition \ref{prop:sign of Uy}, it is easy to see 
\be \label{E:temp-2}
\left |\tilde b(\tau) + \tilde b_2(\tau)  \right| \lesssim  \delta^{-\frac 74} e^{-\frac {3(1-|\tau|)}\delta}, 
\ee
where 
\[
\tilde b_2(\tau) := 2 \int_{-\frac 5\delta}^{\frac 5\delta} \left(\p_{x_2} U  (x_1, \tfrac {1-\tau}\delta)\right)^2 - \left(\p_{x_2} U(x_1, \tfrac {1+\tau}\delta)\right)^2 \,  \diff x_1.
\]
Since $\tilde b_2$ is also odd, we consider $\tau\in (0, \frac 13)$. Using Proposition \ref{prop:sign of Uy} once more, along with \eqref{E:p2U}, we  compute that
\begin{align*}
\tilde b_2(\tau) & = - 4 \int_{-\frac 5\delta}^{\frac 5\delta} \int_{\frac {1-\tau}\delta}^{\frac {1+\tau}\delta}  (\p_{x_2} U \p_{x_2}^2 U)  \, \diff x_2 \, \diff x_1 \\
& =  - 4 \int_{-\frac 5\delta}^{\frac 5\delta} \int_{\frac {1-\tau}\delta}^{\frac {1+\tau}\delta}  \sin{(\theta)} U_r \left(\sin^2{\theta} U_{rr} + \frac {\cos^2{\theta}}r U_r \right ) \, \diff x_2 \, \diff x_1 \\
& \eqsim  \int_{-\frac 5\delta}^{\frac 5\delta} \int_{\frac {1-\tau}\delta}^{\frac {1+\tau}\delta} r^{-1} e^{-2r} \, \diff x_2 \, \diff x_1,
\end{align*}
where we used the fact that $0 < \sin \theta \eqsim 1$ in this integral region. Let 
\[
S = \left\{ x : |x_1| < \tfrac 5\delta, \ |x_2-\tfrac 1\delta| < \tfrac \tau{\delta}, \ |x| < \tfrac {1+\tau}\delta \right\},
\]
which has the polar coordinates representation 
\[
S = \left\{ (r, \theta) : r \in (\tfrac {1-\tau}\delta, \tfrac {1+\tau}\delta), \ \theta \in \left(\beta(r), \pi -\beta(r) \right) \right\},
\] 
where, because we are restricting to $\tau \in (0, \frac 13)$, 
\[
\beta (r) = \arcsin{\left( \tfrac {1-\tau}{\delta r} \right)}, \quad \frac \pi2 - \beta(r) \eqsim (\delta r - 1+\tau)^{\frac 12}. 
\]
Therefore we have 
\begin{align*}
\tilde b_2 (\tau) \gtrsim & \int_S r^{-1} e^{-2r} \, \diff x_2 \, \diff x_1 = \int_{\frac {1-\tau}\delta}^{\frac {1+\tau}\delta} \int_{\beta(r)}^{\pi -\beta(r)} e^{-2r} \, \diff \theta \, \diff r= \int_{\frac {1-\tau}\delta}^{\frac {1+\tau}\delta} \left(\pi -2\beta(r)\right) e^{-2r} \, \diff r\\
\eqsim & \int_{\frac {1-\tau}\delta}^{\frac {1+\tau}\delta} \left( \delta r - (1-\tau) \right)^{\frac 12} e^{-2r}\, \diff r = \delta^{\frac 12} e^{ -\frac {2(1-\tau)}\delta} \int_0^{\frac {2\tau}\delta} ( r')^{\frac 12} e^{- 2r'} \, \diff r'. 
\end{align*}
This implies that, for $\frac \tau\delta \le 1$, 
\[
\tilde b_2(\tau) \gtrsim \tau^{\frac 12}  e^{-\frac 2\delta},
\]
and thus we obtain from \eqref{E:temp-2} that there exists $C>0$ independent of $\delta>0$ such that 
\[
\tilde b(\tau_0) < 0, \quad \tau_0 =C \delta^{-\frac 72} e^{-\frac 2{\delta}}. 
\]

From the oddness of $\tilde b_2$, we can then conclude that there exists $\tilde \tau$ with $|\tilde \tau| \lesssim \delta^{-\frac 72} e^{-\frac 2{\delta}}$ and such that $(\tilde \tau, \tilde v(\tilde \tau), \Gammast (\tilde \tau))$ is a solution to \eqref{contraction map}, and thus corresponds to a solution  to the stationary capillary-gravity wave problem.  The stream function is given by
\be \label{E:Psi-Esti}
\Psi =  \left(U(\placeholder -\frac {\tilde \tau}\delta e_2) - U (\placeholder -\frac {\tilde \tau}\delta e_2)_{\bc}+ \tilde v (\tilde \tau)\right) \circ \left( \tfrac 1\delta \left( \id + \tilde \Gamma (\tilde \tau) \right)^{-1}\right), 
\ee
defined on 
\[ 
\Omega = \left(\id + \tilde \Gamma (\tilde \tau) \right) (\{ |x_2| <1\}). 
\]
From the  estimate $|\tilde \tau| \lesssim\delta^{-\frac 72} e^{-\frac 2{\delta}}$,  
Corollary \ref{C:Ubc}, and Lemma \ref{reduction lemma}, we have 
\[
\left|\tilde v (\tilde \tau) \circ \left( \tfrac 1\delta \left( \id + \tilde \Gamma (\tilde \tau) \right)^{-1}\right) \right|_{H^{k_0} (\Omega)} \lesssim \delta^{1-k_0} |\tilde v(\tilde \tau)|_{H^{k_0} (\slab_\delta)} \lesssim \delta^{1 -2k_0} e^{-\frac 2\delta},
\]
and 
\begin{align*}
&\left| \left(U(\placeholder - \frac {\tilde \tau}\delta e_2) - U (\placeholder -\frac {\tilde \tau}\delta e_2)_{\bc}\right) \circ \left( \tfrac 1\delta \left(\id + \tilde \Gamma (\tilde \tau) \right)^{-1}\right) - \tilde \Psi_0 (\tfrac \placeholder\delta) \right|_{H^{k_0} (\Omega)} \\
&\qquad  \le  \left|\left(U(\placeholder - \frac {\tilde \tau}\delta e_2) - U (\placeholder - \frac {\tilde \tau}\delta e_2)_{\bc} -\tilde \Psi_0 \right) \circ \left( \tfrac 1\delta \left(\id + \tilde \Gamma (\tilde \tau) \right)^{-1}\right)  \right|_{H^{k_0} (\Omega)} \\
& \qquad\qquad +  \left|\tilde \Psi_0 \circ \left( \tfrac 1\delta \left(\id + \tilde \Gamma (\tilde \tau) \right)^{-1}\right) -\tilde \Psi_0 (\tfrac \placeholder\delta) \right|_{H^{k_0} (\Omega)} \\
& \qquad \lesssim  \delta^{\frac 74-k_0} e^{-\frac 2\delta} + |\tilde \Psi_0 (\tfrac \placeholder\delta)|_{H^{k_0+1} (\R^2)} |\Gammast(\tilde \tau)|_{H^{k_0 -\frac 12} (\R)} \lesssim \delta^{1 -2k_0} e^{-\frac 2\delta},
\end{align*}
where 
\[
\tilde \Psi_0(x) = U(x -\frac {\tilde \tau}\delta e_2) - U (x_1, \frac {2-\tilde \tau}\delta -x_2) - U (x_1, - \frac {2+\tilde \tau}\delta -x_2). 
\]
The desired estimate on $\Psi$ in Theorem \ref{main euler theorem} follows immediately. 

Finally, the corresponding free surface profile $\eta$ is given by 
\[
\eta = \Gammast (\tilde \tau) \circ \left(\id + \tilde \Gamma_1(\tilde \tau, \placeholder, 1)\right)^{-1},
\]
which clearly satisfies 
\[
|\eta|_{H^{k_0} (\R)} \lesssim 
\delta^{1 - k_0} e^{-\frac 2\delta}.
\] 
Using \eqref{E:temp-0} and Lemma \ref{reduction lemma}, it is straightforward to identify the leading order term of $\eta$ coinciding with that of $\Gammast (\tilde \tau)$ and to obtain the same remainder estimate much as in the above procedure for $\Psi$. 
This completes the proof of the main theorem. 
\end{proof}

\section*{Acknowledgements}
The authors wish to thank one of the referees whose extraordinarily thorough reading and many suggestions lead to substantial  improvements.

\bibliographystyle{siam}
\bibliography{spike}

\begin{thebibliography}{10}

\bibitem{MR3626577}
{\sc S.~Adachi, M.~Shibata, and T.~Watanabe}, {\em Global uniqueness results
  for ground states for a class of quasilinear elliptic equations}, Kodai Math.
  J., 40 (2017), pp.~117--142.

\bibitem{MR734575}
{\sc H.~Berestycki, T.~Gallou\"et, and O.~Kavian}, {\em {\'Equations de champs
  scalaires euclidiens non lin\'eaires dans le plan}}, C. R. Acad. Sci. Paris
  S\'er. I Math., 297 (1983), pp.~307--310.

\bibitem{MR695535}
{\sc H.~Berestycki and P.-L. Lions}, {\em Nonlinear scalar field equations.
  {I}. {E}xistence of a ground state}, Arch. Rational Mech. Anal., 82 (1983),
  pp.~313--345.

\bibitem{chen2019existence}
{\sc R.~M. Chen, S.~Walsh, and M.~H. Wheeler}, {\em Existence, nonexistence,
  and asymptotics of deep water solitary waves with localized vorticity}, Arch.
  Rational Mech. Anal., 234 (2019), pp.~595--633.

\bibitem{MR2224508}
{\sc C.~Chicone}, {\em Ordinary differential equations with applications},
  vol.~34 of Texts in Applied Mathematics, Springer, New York, second~ed.,
  2006.

\bibitem{constantin2001edge}
{\sc A.~Constantin}, {\em Edge waves along a sloping beach}, J. Phys. A, 34
  (2001), p.~9723.

\bibitem{constantin2010particle}
\leavevmode\vrule height 2pt depth -1.6pt width 23pt, {\em On the particle
  paths in solitary water waves}, Q. Appl. Math., 68 (2010), pp.~81--90.

\bibitem{constantin2011dynamical}
\leavevmode\vrule height 2pt depth -1.6pt width 23pt, {\em A dynamical systems
  approach towards isolated vorticity regions for {T}sunami background states},
  Arch. Ration. Mech. Anal., 200 (2011), pp.~239--253.

\bibitem{constantin2011book}
\leavevmode\vrule height 2pt depth -1.6pt width 23pt, {\em Nonlinear water
  waves with applications to wave-current interactions and tsunamis}, vol.~81
  of CBMS-NSF Regional Conference Series in Applied Mathematics, Society for
  Industrial and Applied Mathematics (SIAM), Philadelphia, PA, 2011.

\bibitem{constantin2007particle}
{\sc A.~Constantin and J.~Escher}, {\em Particle trajectories in solitary water
  waves}, Bull. Am. Math. Soc, 44 (2007), pp.~423--431.

\bibitem{constantin2004exact}
{\sc A.~Constantin and W.~Strauss}, {\em Exact steady periodic water waves with
  vorticity}, Comm. Pure Appl. Math., 57 (2004), pp.~481--527.

\bibitem{constantin2009pressure}
{\sc A.~Constantin and W.~Strauss}, {\em Pressure beneath a {S}tokes wave},
  Comm. Pure Appl. Math., 63 (2010), pp.~533--557.

\bibitem{constantin2016global}
{\sc A.~Constantin, W.~Strauss, and E.~V\u{a}rv\u{a}ruc\u{a}}, {\em Global
  bifurcation of steady gravity water waves with critical layers}, Acta Math.,
  217 (2016), pp.~195--262.

\bibitem{craig2002nonexistence}
{\sc W.~Craig}, {\em Non-existence of solitary water waves in three
  dimensions}, R. Soc. Lond. Philos. Trans. Ser. A Math. Phys. Eng. Sci., 360
  (2002), pp.~2127--2135.
\newblock Recent developments in the mathematical theory of water waves
  (Oberwolfach, 2001).

\bibitem{crapper1957exact}
{\sc G.~Crapper}, {\em An exact solution for progressive capillary waves of
  arbitrary amplitude}, J. Fluid Mech., 2 (1957), pp.~532--540.

\bibitem{dasilva1988steep}
{\sc A.~T. Da~Silva and D.~Peregrine}, {\em Steep, steady surface waves on
  water of finite depth with constant vorticity}, J. Fluid Mech., 195 (1988),
  pp.~281--302.

\bibitem{ehrnstrom2011steady}
{\sc M.~Ehrnstr\"{o}m, J.~Escher, and E.~Wahl\'{e}n}, {\em Steady water waves
  with multiple critical layers}, SIAM J. Math. Anal., 43 (2011),
  pp.~1436--1456.

\bibitem{MR2409513}
{\sc M.~Ehrnstr{\"o}m and G.~Villari}, {\em Linear water waves with vorticity:
  rotational features and particle paths}, J. Differential Equations, 244
  (2008), pp.~1888--1909.

\bibitem{EW14ARMA}
{\sc M.~Ehrnstr{\"o}m and E.~Wahl{\'e}n}, {\em Trimodal steady water waves},
  Arch. Rational Mech. Anal.,  (2014), pp.~1--23.

\bibitem{filippov1960vortex}
{\sc I.~G. Filippov}, {\em Solution of the problem of the motion of a vortex
  under the surface of a fluid, for {F}roude numbers near unity}, J. Appl.
  Math. Mech., 24 (1960), pp.~698--716.

\bibitem{filippov1961motion}
\leavevmode\vrule height 2pt depth -1.6pt width 23pt, {\em On the motion of a
  vortex below the surface of a liquid}, J. Appl. Math. Mech., 25 (1961),
  pp.~357--365.

\bibitem{gerstner1809theorie}
{\sc F.~Gerstner}, {\em Theorie der wellen}, Annalen der Physik, 32 (1809),
  pp.~412--445.

\bibitem{GW99}
{\sc C.~Gui and J.~Wei}, {\em Multiple interior peak solutions for some
  singularly perturbed {N}eumann problems}, J. Differential Equations, 158
  (1999), pp.~1--27.

\bibitem{henry2008gerstner}
{\sc D.~Henry}, {\em On {G}erstner's water wave}, J. Nonlinear Math. Phy., 15
  (2008), pp.~87--95.

\bibitem{hur2012no}
{\sc V.~M. Hur}, {\em No solitary waves exist on 2{D} deep water},
  Nonlinearity, 25 (2012), pp.~3301--3312.

\bibitem{kinnersley1976exact}
{\sc W.~Kinnersley}, {\em Exact large amplitude capillary waves on sheets of
  fluid}, J. Fluid Mech., 77 (1976), pp.~229--241.

\bibitem{MR3808597}
{\sc V.~Kozlov and E.~Lokharu}, {\em {$N$}-modal steady water waves with
  vorticity}, J. Math. Fluid Mech., 20 (2018), pp.~853--867.

\bibitem{MR969899}
{\sc M.~K. Kwong}, {\em Uniqueness of positive solutions of {$\Delta
  u-u+u^p=0$} in {${\bf R}^n$}}, Arch. Rational Mech. Anal., 105 (1989),
  pp.~243--266.

\bibitem{le2018existence}
{\sc H.~Le}, {\em On the existence and instability of solitary water waves with
  a finite dipole}, SIAM J. Math. Anal., 51 (2019), pp.~4074--4104.

\bibitem{levicivita1925determination}
{\sc T.~Levi-Civita}, {\em D{\'e}termination rigoureuse de ondes permanentes
  d'ampleur finie}, Ann. Math., 93 (1925), pp.~264--314.

\bibitem{MR1639159}
{\sc Y.~Li and L.~Nirenberg}, {\em The {D}irichlet problem for singularly
  perturbed elliptic equations}, Comm. Pure Appl. Math., 51 (1998),
  pp.~1445--1490.

\bibitem{MR3436246}
{\sc C.~I. Martin and B.-V. Matioc}, {\em Gravity water flows with
  discontinuous vorticity and stagnation points}, Commun. Math. Sci., 14
  (2016), pp.~415--441.

\bibitem{nekrasov1951exact}
{\sc A.~I. Nekrasov}, {\em The exact theory of steady waves on the surface of a
  heavy fluid}, Izdat. Akad. Nauk SSSR, Moscow, 1951.

\bibitem{NiTa91}
{\sc W.-M. Ni and I.~Takagi}, {\em On the shape of least-energy solutions to a
  semilinear {N}eumann problem}, Comm. Pure Appl. Math., 44 (1991),
  pp.~819--851.

\bibitem{ni1998location}
{\sc W.-M. Ni, I.~Takagi, and J.~Wei}, {\em On the location and profile of
  spike-layer solutions to a singularly perturbed semilinear {D}irichlet
  problem: intermediate solutions}, Duke Math. J., 94 (1998), pp.~597--618.

\bibitem{ni1995location}
{\sc W.-M. Ni and J.~Wei}, {\em On the location and profile of spike-layer
  solutions to singularly perturbed semilinear {D}irichlet problems}, Comm.
  Pure Appl. Math., 48 (1995), pp.~731--768.

\bibitem{shatah2013travelling}
{\sc J.~Shatah, S.~Walsh, and C.~Zeng}, {\em Travelling water waves with
  compactly supported vorticity}, Nonlinearity, 26 (2013), pp.~1529--1564.

\bibitem{simmen1985steady}
{\sc J.~A. Simmen and P.~Saffman}, {\em Steady deep-water waves on a linear
  shear current}, Studies in Applied Mathematics, 73 (1985), pp.~35--57.

\bibitem{sun1997analytical}
{\sc S.~M. Sun}, {\em Some analytical properties of capillary-gravity waves in
  two-fluid flows of infinite depth}, Proc. Roy. Soc. London Ser. A, 453
  (1997), pp.~1153--1175.

\bibitem{terkrkorov1958vortex}
{\sc A.~M. Ter-Krikorov}, {\em Exact solution of the problem of the motion of a
  vortex under the surface of a liquid}, Izv. Akad. Nauk SSSR Ser. Mat., 22
  (1958), pp.~177--200.

\bibitem{toland1996stokes}
{\sc J.~F. Toland}, {\em Stokes waves}, Topol. Methods Nonlinear Anal., 7
  (1996), pp.~1--48.

\bibitem{varholm2016solitary}
{\sc K.~Varholm}, {\em Solitary gravity-capillary water waves with point
  vortices}, Discrete Contin. Dyn. Syst., 36 (2016), pp.~3927--3959.

\bibitem{varholmthesis}
\leavevmode\vrule height 2pt depth -1.6pt width 23pt, {\em On Steady Water
  Waves with Stagnation Points}, PhD thesis, Norwegian University of Science
  and Technology, 2019.

\bibitem{wahlen2009steady}
{\sc E.~Wahl\'{e}n}, {\em Steady water waves with a critical layer}, J.
  Differential Equations, 246 (2009), pp.~2468--2483.

\bibitem{wheeler2018integral}
{\sc M.~H. Wheeler}, {\em Integral and asymptotic properties of solitary waves
  in deep water}, Comm. Pure Appl. Math., 71 (2018), pp.~1941--1956.

\end{thebibliography}

\end{document}